\documentclass[a4paper,american,reqno]{amsart}

\usepackage[utf8]{inputenc}
\usepackage[T1]{fontenc}
\usepackage[a4paper]{geometry}
\usepackage[american]{babel}
\usepackage{lmodern}
\usepackage{hyperref}
\usepackage{amsmath}
\usepackage{amssymb}
\usepackage{amsthm}
\usepackage{graphicx}
\usepackage{listings}
\usepackage{subcaption}
\usepackage{algorithm}
\usepackage{algpseudocode}
\usepackage{xspace}
\usepackage{todonotes}
\usepackage{csquotes}
\usepackage[font=small]{caption}
\usepackage{nameref}
\usepackage{cleveref}
\usepackage{booktabs}
\usepackage{natbib}
\usepackage{float}
\renewcommand{\:}{\colon}
\newcommand{\la}{\langle}
\newcommand{\ra}{\rangle}

\newcommand{\RR}{\mathbb{R}}
\newcommand{\Ut}{\ensuremath{U_t}}

\newtheorem{theorem}{Theorem}[section]
\newtheorem{proposition}[theorem]{Proposition}
\newtheorem{corollary}[theorem]{Corollary}
\newtheorem{lemma}[theorem]{Lemma}
\theoremstyle{definition}
\newtheorem{Assumption}{Assumption}

\theoremstyle{remark}
\newtheorem*{remark}{Remark}

\title[Stabilization of PDEs by Sequential Action Control]{Stabilization of Partial Differential Equations by Sequential Action Control}

\author[]{Yan Brodskyi$^{1}$, Falk M. Hante$^{1}$,~Arno Seidel$^{2}$}
\thanks{$^1$ Department of Mathematics, Humboldt-Universit\"at zu Berlin,
Unter den Linden 6, 10099 Berlin, Germany (\url{yan.brodskyi@hu-berlin.de}, \url{falk.hante@hu-berlin.de}); $^2$Friedrich-Alexander-Universit\"at Erlangen-N\"urn\-berg, Germany (\url{arno.seidel@fau.de})}

\date{August 26, 2022}

\begin{document}

\maketitle
\begin{abstract}
{We present a framework of sequential action control (SAC) for stabilization of systems of partial differential equations which can be posed as abstract semilinear control problems in Hilbert spaces. We follow a late-lumping approach and show that the control action can be explicitly obtained from variational principles using adjoint information. Moreover, we analyze the closed-loop system obtained from the SAC feedback for the linear problem with quadratic stage costs. We apply this theory to a prototypical example of an unstable heat equation and provide numerical results as the verification and demonstration of the framework.}
\end{abstract}

\section{Introduction}
With the rising performance of embedded and networked systems the stabilization of processes governed by partial differential equations 
(PDEs) in real time and under uncertainties becomes an important topic in control theory. Here we consider processes predicted by 
semilinear PDEs of evolution type with distributed control given in an abstract sense by the operator differential equation
\begin{equation*}
  \dot{y}(t)=Ay(t) + f(t,y(t)) u(t) \in H,\quad t>0
\end{equation*}
with a state $y$ and control $u$ in possibly infinite-dimensional spaces $H$ and $U$, respectively, and satisfying the initial condition
\begin{equation*}
 y(0)=y_0
\end{equation*}
for some initial state $y_0 \in H$. The goal is to guarantee asymptotic stabilization of $y$ at the origin, or more generally, to consider nonlinear path following for a given desired trajectory $y_d(t)$, $t>0$. Similar problems also arise in optimal control of PDEs, if one wishes to exploit the Turnpike property, saying that in large optimization time horizons, the optimal solution remains exponentially close to an optimal stationary solution for most of the time. This property has been proven for various optimal control systems, see \cite{GrueneEtAl2019, GugatHante, TrelatZhang2018} and the references therein.

Our investigations here concern a moving horizon strategy 
for the control design which allows to include measurements to account for uncertainties in an online fashion, e.g., for the typically not exactly known initial 
state $y_0$ for the next horizon. A challenging application, for instance, is the stabilization of gas networks, where the realized demand has some variations over a
predicted one during intra-day operation, see e.g., \cite{HanteEtAl2017}.

In this context, the most prominent control design is (nonlinear) model predictive control (MPC), where future control action is obtained from the solution of a dynamic optimization problem. Concerning problems involving PDEs, stability analysis for the closed loop with a focus on late lumping is carried out, for example, in 
\cite{AllaVolkwein2015, Altmueller2014, AltmuellerGruene2012, DubljevicEtAl2006, FleigGruene2018, ItoKunisch2002, GrueneEtAl2019}. In this work, we consider a moving horizon strategy differing from MPC in avoiding  to numerically compute the solution to a time-depend optimal control problem in each step. Instead, a piecewise constant control action is computed by selecting a control value and an application time 
that ensures a certain decrease of the current stage costs. This principle has been introduced for nonlinear problems with ordinary differential equations as sequential action control (SAC) in \cite{AnsariMurphey2016}. It relies on representations of the so-called mode insertion gradient in \cite{Axelsson2008} using adjoint information originating in optimization of switched dynamical systems and the needle variation in \cite{PontryaginEtAl1962} as used in the derivation of Pontryagin's principle. 

Our contribution here is to extend the idea of sequential action control to the above PDE setting in a Hilbert space framework. We show that the control action can be explicitly obtained from variational principles using an adjoint PDE. Moreover, for the linear case this allows us to show that the closed-loop performance for quadratic stage costs can in first order be characterized by a system under a particular linear feedback. Using this, we derive mesh independent stabilizing properties of this method prototypically for a benchmark problem from \cite{AltmuellerGruene2012}. A finite-dimensional controller can then be obtained using Galerkin approximations. Unlike applying the theory of \cite{AnsariMurphey2016} to a finite-dimensional (lumped) approximation of the PDE this allows independent discretizations for the forward and the adjoint problem. For this linear problem we also provide a numerical study for both full domain and subdomain control with the most important parameters in this framework. Further, we numerically investigate the robustness of the closed-loop stabilization by SAC with respect to random disturbances and compare SAC with the standard linear-quadratic regulator (LQR) method. SAC turns out to be faster in all cases and more robust for the case of subdomain control with small disturbances.

The remaining article is organized as follows. In Section~\ref{sec:framework}, we introduce the Hilbert space framework of SAC. In Section~\ref{sec:linear-quad}, we consider the important special case of linear quadratic problems, for which we characterize the closed-loop system for the stability analysis of SAC actions. Moreover, we analyze the performance of SAC for an unstable heat equation. In Section ~\ref{sec:numerical}, we show the corresponding numerical results including the comparison with LQR. In Section~\ref{sec:conclusion}, we provide concluding remarks.

\section{A Hilbert space framework for SAC}\label{sec:framework}

Let $H$ and $U$ be real (separable) Hilbert spaces with inner products $\la \cdot, \cdot \ra_H$ and $\la \cdot, \cdot \ra_U$, and its topological dual spaces $H^*$ and $U^*$, respectively. Further, let $\mathfrak{L} (U,H)$ denote the space of bounded linear operators from $U$ to $H$. In this setting we consider a moving horizon strategy with the stage problem
\begin{equation}\label{eq:stageproblem}
 \begin{aligned}
  \min~J_1(u) = &\int_0^T l_1(y(s))\,ds + m(y(T))\\
  \text{s.t.}\qquad&\dot{y}(t)=Ay(t) + f(t,y(t)) u(t),~t \in (0,T),\quad y(0) = y_0,\\
  &u(t) \in U,~t \in (0,T),  
 \end{aligned}
\end{equation}
where $A$ is a (possibly unbounded) linear operator on $H$, $l_1$ and $m$ are (possibly nonlinear) functionals on $H$, $f: (0,T)\times H \to \mathfrak{L} (U,H)$, where $\mathfrak{L} (U,H)$ is space of bounded linear operators, $y_0$ is a fixed look ahead initial state in $H$, $T$ is a fixed time horizon and the minimization is with respect to all $u \in \Ut=L^p(0,T;U)$ for some $p\geq 1$. 

In order to fix notation in what follows, we denote, whenever $X$ and $Y$ are Hilbert spaces, for $x \in X$ the dual pairing by $\la x^*, x \ra_{X^*,\; X} = x^*(x)$, and for any bounded linear operator $B$ from $X$ to $Y$ by $ \|B\|_{op} = \inf \{c:\|Bx\|_Y \le c \|x\|_X ~ \mbox{for all} ~ x \in X\} $ the operator norm, by $B^*\: Y^* \rightarrow X^*$ defined as $\la B^*(y^*), x \ra_{X^*,\; X} = \la y^*, Bx \ra_{Y^*,\; Y}$ the topological adjoint operator, and by $B^{\star}\: Y \rightarrow X$ defined as $\la B^{\star}y, x \ra_X = \la y, Bx \ra_Y$ the Hilbert space adjoint operator. Further notations and technical assumptions on the operators and functionals in \eqref{eq:stageproblem} will be introduced in the sequel.

\subsection{The SAC principle}\label{subsec:principle}
Given a reference control $u_1 \in \Ut$ ($u_1 \equiv 0$ is a feasible choice), let $u_{\lambda,\tau,v}$ denote the needle variation defined by
\begin{equation*}
 u_{\lambda,\tau,v} = \begin{cases} u_1(t), & t \notin [\tau-\frac{\lambda}{2},\tau+\frac{\lambda}{2}]\\ v, & t \in [\tau-\frac{\lambda}{2},\tau+\frac{\lambda}{2}]. \end{cases}
\end{equation*}
Further, let $\frac{dJ_1}{d\lambda^+}(\tau,v)$ denote the sensitivity
\begin{equation}\label{eq:mig}
 \frac{dJ_1}{d\lambda^+}(\tau,v) = \lim_{\lambda \downarrow 0} \frac{J_1(u_{\lambda,\tau,v})-J_1(u_1)}{\lambda}
\end{equation}
for the current stage \eqref{eq:stageproblem}. The principle of SAC relies on choosing the control values $u_{\text{opt}}(\tau)$ as
\begin{equation}\label{eq:l2problem}
\begin{aligned}
u_{\text{opt}}(\tau) &:= \mbox{argmin}_{u \in U}~l_2(u; \tau):=\frac12\left[ \frac{dJ_1}{d\lambda^+}(\tau,u)-\alpha_d \right]^2+\frac12\la u,R u\ra_U,
\end{aligned}
\end{equation}
for some application time $\tau \in (0,T)$, where $\alpha_d < 0$, is to be chosen appropriately in order to obtain a sufficiently large reduction for the current prediction according to \eqref{eq:stageproblem} when $u_{\text{opt}}(\tau)$ is applied for the current stage on the interval $[\tau-\frac{\bar{\lambda}}{2},\tau+\frac{\bar{\lambda}}{2}]$ with a suitably chosen duration $\bar{\lambda}>0$ until a new control has been computed for a shifted prediction horizon. Here $\bar{\lambda}$ corresponds to the length of the application time of the control $u_{\text{opt}}(\tau)$. The operator $R$ in \eqref{eq:l2problem} can be suitably chosen as a regularization parameter or to model control costs in this process.

Figure~\ref{fig:A03-RH-Strategien} shows the schematic overview of the SAC principle. In general, it consists from the following steps:
\begin{enumerate}
    \item predict the nominal dynamics of the state and the adjoint for the reference control $u_1$;
    \item compute the control values $u_{\text{opt}}(\tau)$ according to the SAC principle \eqref{eq:l2problem}; 
    \item select an application time $\tau$ and the action’s duration $\bar{\lambda}$;
    \item apply $u_{\bar{\lambda}, \tau, u_{\text{opt}}(\tau)}$ and repeat the process at the next sample time, $t_0 = t_0 + t_s$.
\end{enumerate}
\begin{figure}[!t]
\centering
\captionsetup{width=0.8\linewidth}
\includegraphics[width=0.6\linewidth]{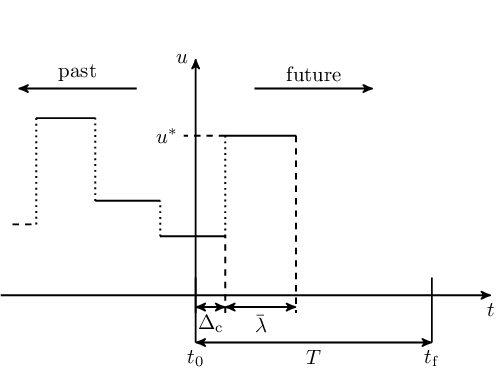}
\caption{A schematic illustration of the SAC principle \cite{AnsariMurphey2016}.}
\label{fig:A03-RH-Strategien}
\end{figure}
For details concerning all of these steps we make the following hypotheses. 

\begin{Assumption}\label{ass:main}
The operator $A$ is densely defined on $D(A) \subset H$ and generates a strongly continuous semigroup $S(t)=e^{tA}$ on $H$. The cost functionals $l_1\: H \to \RR$ and $m\: H \to \RR$ are continuously differentiable in the sense of Fr\'{e}chet and uniformly bounded. For every $t \in (0,T), ~ f(t,\cdot)\: H \to \mathfrak{L} (U,H)$ is continuously Fr\'{e}chet differentiable. Furthermore, $f$ is linearly bounded or $T$ is sufficiently small. The operator $R : U \to U$ is bounded, self-adjoint and satisfies, for some $\gamma_R>0$, the coercivity estimate 
$\la R u,u\ra \geq \gamma_R \|u\|_U^2$ for all $u \in U$. 
\end{Assumption}

Under the above assumptions, the abstract initial value problem in \eqref{eq:stageproblem} 
\begin{equation}\label{eq:forwardproblem}
\dot{y}(t)=Ay(t) + f(t,y(t)) u(t),~u(t) \in U,\quad t \in (0,T), \quad y(0) = y_0
\end{equation}
has a unique solution $y \in C([0,T];H)$ in the mild sense, i.e., 
\begin{equation}\label{eq:forwardproblem-mild}
y(t) = e^{tA} y_0+\int_0^t e^{(t-s)A}f(s,y(s)) u(s)\,ds,\quad t \in [0,T] 
\end{equation}
for any $u \in \Ut$ and any $y_0 \in H$. Moreover, the following adjoint problem
\begin{equation}\label{eq:adjointproblem}
\begin{aligned}
 \dot{p}(t) &= -A^* p(t) -((f(t,y(t))u(t))_{y})^{*} p(t) - (l_1)_y(y(t))\\
 p(T) &= m_y(y(T))
\end{aligned}
\end{equation}
also has a unique solution $p \in C([0,T];H^*)$ in a mild sense, i.e., 
\begin{equation}\label{eq:adjointproblem-mild}
p(t) = e^{(T-t)A^*} p(T)+\int_t^T e^{(s - t)A^*}(((f(s,y(s))u(s))_{y})^{*}p(s) + (l_1)_y(y(s)))\,ds,\quad t \in [0,T] 
\end{equation} 
for any $u \in \Ut$. Here $f(t,y(t))u(t))_{y}$ is a gradient of $f(t,y(t))u(t))$ with respect to $y$. $A^{*}$ is the generator of $C_0$-semigroup on $H^*$, see \cite[Chapter 1, Theorem 10.6]{Pazy}. The existence and uniqueness of the mild solution of \eqref{eq:adjointproblem} is given by \cite[Chapter 6, Theorem 1.2]{Pazy} due to the continuous Fr\'{e}chet differentiability assumptions on $f$ and linearity of the adjoint problem in $p$. In addition, the problem \eqref{eq:l2problem} admits an explicit solution, as stated in the following theorem.

\begin{theorem}\label{thm:main} Let $y \in C([0,T];H)$ be a solution of the state equation in \eqref{eq:stageproblem} with $u$ being a given reference control $u_1 \in C(0,T;U)$. Further, let $p \in C([0,T];H^*)$
be the corresponding solution of the adjoint problem \eqref{eq:adjointproblem}. With $P \in C([0,T];U^*)$ defined by
\begin{equation}\label{eq:Pdef}
 P(t)\: w \mapsto \la p(t), f(t,y(t)) w\ra_{H^{*}, \; H},\quad t \in [0,T],
\end{equation}
for which the operator $P^{\star}(t)P(t)+R^{\star}$ has a bounded inverse. And with 
\begin{equation}\label{eq:pidef}
\Pi(t) = (P^{\star}(t)P(t)+R^{\star})^{-1}, \quad t \in [0,T],
\end{equation}
the optimal solution $u_{\text{opt}}(\tau)$ of the problem \eqref{eq:l2problem} is given by
\begin{equation}\label{eq:soll2problem}
 u_{\text{opt}}(\tau) = \Pi(\tau)[P^{\star}(\tau)P(\tau)u_1(\tau)+P^{\star}(\tau)\alpha_d]
\end{equation}
for all $\tau \in [0,T]$.
\end{theorem}
\begin{proof}
Under the Assumption~\ref{ass:main}, we can use results in \cite[Theorem~10]{RuefflerHante2016} to obtain that the limit in \eqref{eq:mig}
exists and that it is given by
\begin{equation}\label{eq:migrep}
 \frac{dJ_1}{d\lambda^+}(\tau,u)=\la p(\tau),f(\tau,y(\tau))(u(\tau)-u_1(\tau))\ra_{H^{*},\; H}
\end{equation}
with $p$ satisfying \eqref{eq:adjointproblem}. Moreover, as a sum of convex functions, $l_2(\tau,u(\tau))$ is convex as a function of $u$. With this, necessary optimality conditions become sufficient. Classical optimality conditions (\cite[Theorem~1.46]{Hinze_Pinnau}) now yield that
\begin{equation}\label{eq:optcond2}
\frac{\mathfrak{d} l_2}{\mathfrak{d}u_{\text{opt}}}(u_{\text{opt}})\eta = \lim_{\varepsilon \downarrow 0} \frac1\varepsilon[l_2(u_{\text{opt}}+ \varepsilon\eta)-l_2(u_{\text{opt}})] = 0.
\end{equation}
With \eqref{eq:migrep}, we obtain that for all $\tau \in [0,T]$ and all $\eta \in C(0,T;U)$
\begin{equation*}
\begin{aligned}
 l_2(u_{\text{opt}}(\tau)+ \varepsilon\eta(\tau)) =~&\frac12 \left[ \la p(\tau),f(\tau,y(\tau))(u_{\text{opt}}(\tau)+\varepsilon\eta-u_1(\tau))\ra_{H^{*},\; H}-\alpha_d\right]^2\\
 &\quad + \frac12\la u_{\text{opt}}(\tau)+\varepsilon\eta(\tau),R(u_{\text{opt}}(\tau)+\varepsilon\eta(\tau))\ra_U\\
\end{aligned}
\end{equation*}
and, with a short calculation, that
\begin{equation}\label{eq:underint}
\begin{aligned}
 l_2(u_{\text{opt}}(\tau)+ \varepsilon\eta(\tau))-l_2(u_{\text{opt}}(\tau)) =~&\la p(\tau),f(\tau,y(\tau))(u_{\text{opt}}(\tau)-u_1(\tau)) \ra_{H^{*},\; H} \la p(\tau),\varepsilon f(y(\tau),\tau) \eta(\tau) \ra_{H^{*},\; H} \\
 & + \frac12 \la p(\tau),\varepsilon f(y(\tau),\tau) \eta(\tau) \ra_{H^{*},\; H}^2 - \alpha_d\la p(\tau),\varepsilon f(y(\tau),\tau) \eta(\tau) \ra_{H^{*},\; H} \\
 & + \frac{\varepsilon}{2} \la u_{\text{opt}}(\tau),R \eta(\tau) \ra_U + \frac{\varepsilon}{2} \la \eta(\tau),R(u_{\text{opt}}(\tau)+\varepsilon\eta(\tau)) \ra_U.
\end{aligned}
\end{equation}
By linearity, the right hand side in \eqref{eq:underint} converges uniformly as $\varepsilon \to 0$. Hence, we can take the limit in \eqref{eq:optcond2} and obtain,
for almost every $\tau \in [0,T]$, for all $\xi \in U$
\begin{equation}\label{eq:firstvarpointwise}
 \left[ \la p(\tau),f(\tau,y(\tau))(u_{\text{opt}}(\tau)-u_1(\tau))\ra_{H^{*},\; H} -\alpha_d \right]\la p(\tau),f(\tau,y(\tau))\xi\ra_{H^{*},\; H} + \la R^{\star} u_{\text{opt}}(\tau),\xi\ra_U = 0.
\end{equation}
With $P$ defined as in \eqref{eq:Pdef}, we have
\begin{equation}\label{eq:firstvarpointwiseP}
\begin{aligned}
&\left[ \la p(\tau),f(\tau,y(\tau))(u_{\text{opt}}(\tau)-u_1(\tau))\ra_{H^{*},\; H} -\alpha_d \right]\la p(\tau),f(\tau,y(\tau))\xi\ra_{H^{*},\; H} + \la R^{\star} u_{\text{opt}}(\tau),\xi\ra_U\\
&\quad \quad \quad \quad \quad \quad \quad \quad \quad \quad \quad  = \la P(\tau)(u_{\text{opt}}(\tau)-u_1(\tau))-\alpha_d,P(\tau)\xi\ra_{\RR} + \la R^{\star} u_{\text{opt}}(\tau),\xi\ra_U\\ 
&\quad \quad \quad \quad \quad \quad \quad \quad \quad \quad \quad  = \la P^{\star}(\tau)(P(\tau)(u_{\text{opt}}(\tau)-u_1(\tau))-\alpha_d),\xi\ra_{U} + \la R^{\star} u_{\text{opt}}(\tau),\xi\ra_U\\
&\quad \quad \quad \quad \quad \quad \quad \quad \quad \quad \quad  = \la (P^{\star}(\tau)P(\tau)+R^{\star})u_{\text{opt}}(\tau)-P^{\star}(\tau)P(\tau)u_1(\tau)-P^{\star}(\tau)\alpha_d,\xi\ra_{U}.
\end{aligned}
\end{equation}
Using \eqref{eq:firstvarpointwiseP} in \eqref{eq:firstvarpointwise}, and that it holds for all $\xi \in U$ yields
\begin{equation}\label{eq:operatoreq}
 (P^{\star}(\tau)P(\tau)+R^{\star})u_{\text{opt}}(\tau)-P^{\star}(\tau)P(\tau)u_1(\tau)-P^{\star}(\tau)\alpha_d = 0.
\end{equation}
For all $\tau \in [0,T]$ the operator $P^{\star}(\tau)P(\tau)$ is self-adjoint and under the hypotheses in Assumption~\ref{ass:main} the operator 
$P^{\star}(\tau)P(\tau)+R^{\star}$ is self-adjoint and satisfies
\begin{equation*}
 \la(P^{\star}(\tau)P(\tau)+R)u,u\ra_U = (P(\tau)u)^2 + \la Ru,u\ra_U \geq \gamma_R \|u\|_U^2.
\end{equation*}
Hence, the operator $P^{\star}(\tau)P(\tau)+R^{\star}$ has a bounded inverse for all $\tau \in [0,T]$ and \eqref{eq:soll2problem} 
follows by rearranging terms in \eqref{eq:operatoreq}.
\end{proof}

\noindent The computation of the SAC actions using the adjoint as in Theorem~\ref{thm:main} is highly efficient compared to the computation of the SAC actions using the numerical approximations of the sensitivity. 

\subsection{Action duration}\label{subsec:action}

For $u_{\text{opt}}(\tau)$ given by the SAC principle \eqref{eq:l2problem}, for arbitrary time $\tau \in (0,T)$ we now want to assure that we can find the duration $\bar{\lambda}$, such that the corresponding change in cost $J_1$ will be negative. Since this is not obvious in the infinite dimensional setting, we provide an explicit proof.

\begin{proposition}\label{eq:prop2.2}
Let $y \in C([0,T];H)$ be a solution of the state equation in \eqref{eq:stageproblem}, $p \in C([0,T];H^*)$ be the corresponding solution of the adjoint problem \eqref{eq:adjointproblem} with $u_1 \in C(0,T;U)$ being a given reference control. Furthermore, let $u_{\text{opt}}(\tau)$ be a control obtained by the SAC principle \eqref{eq:l2problem}. Then there exists a parameter $\alpha_d < 0$, such that the mode insertion gradient \eqref{eq:mig} satisfies 
\begin{equation}
    \begin{aligned}
    & \frac{dJ_1}{d\lambda^+}(\tau,u_{\text{opt}}(\tau)) < 0
    \end{aligned}
\end{equation}
for all $\tau \in (0,T)$. 
\end{proposition}
\begin{proof}
Combining \eqref{eq:soll2problem}, \eqref{eq:migrep} and using Assumption~\ref{ass:main}, with the operator norms $\bar{f} = \|f(\tau, y(\tau))\|_{op}$ and $\bar{P} = \|P^{\star}(\tau)P(\tau)\|_{op}$  we have
\begin{equation}
\begin{aligned}
& \frac{dJ_1}{d\lambda^+}(\tau,u_{\text{opt}}(\tau)) = \la p(\tau), f(\tau,y(\tau))\big(\Pi(\tau)[P^{\star}(\tau)P(\tau)u_1(\tau)+P^{\star}(\tau)\alpha_d] - u_1(\tau) \big) \ra_{H^{*},\; H}\\
&\quad \quad \quad \quad \quad \quad = \la p(\tau),f(\tau,y(\tau)) \Pi(\tau) P^{\star}(\tau)\alpha_d \ra_{H^{*},\; H} - \la p(\tau),f(\tau,y(\tau))u_1(\tau) \ra_{H^{*},\; H}\\
&\quad \quad \quad \quad \quad \quad + \la p(\tau), f(\tau,y(\tau))\Pi(\tau)P^{\star}(\tau)P(\tau)u_1(\tau) \ra_{H^{*},\; H} \\
&\quad \quad \quad \quad \quad \quad \le P(\tau)\Pi(\tau)P^{\star}(\tau)\alpha_d +
\|p(\tau)\|_{H^{*}} \bar{f} \|u_1(\tau)\|_U \big( \|\Pi(\tau)\|_{op} \bar{P} + 1 \big).\\
\end{aligned}    
\end{equation}
The operator $\Pi(\tau)$ is positive, as the inverse of a positive operator. This gives us that the operator $P(\tau)\Pi(\tau)P^{\star}(\tau)$  is also positive. By choosing $\alpha_d < 0$ such that 
\begin{equation}
    \begin{aligned}
    & | P(\tau)\Pi(\tau)P^{\star}(\tau)\alpha_d | < \|p(\tau)\|_{H^{*}} \bar{f} \|u_1(\tau)\|_U \big( 1 + \|\Pi(\tau)\|_{op} \bar{P} \big),
    \end{aligned}
\end{equation}
we receive $\frac{dJ_1}{d\lambda^+}(\tau,u_{\text{opt}}(\tau)) < 0$. 
\end{proof}

\begin{corollary}\label{col:2.3}
Let $y \in C([0,T];H)$ be a solution of the state equation in \eqref{eq:stageproblem}, $p \in C([0,T];H^*)$ be the corresponding solution of the adjoint problem \eqref{eq:adjointproblem} with $u_1 \in C(0,T;U)$ being the given reference control. Furthermore, let $u_{\text{opt}}(\tau)$ be a control obtained by the SAC principle \eqref{eq:l2problem}. Then for all $\tau \in (0,T)$ there exists a parameter  $\alpha_d < 0$ and there exists a neighborhood $N(\tau) \subset (0,T)$ around $\tau$, such that the mode insertion gradient \eqref{eq:mig} satisfies  
\begin{equation}
    \begin{aligned}
    & \frac{dJ_1}{d\lambda^+}(t,u_{\text{opt}}(\tau)) < 0
    \end{aligned}
\end{equation}
for all $t \in N(\tau)$.
\end{corollary}
\begin{proof}
Continuity of the state $y(t)$, the corresponding adjoint $p(t)$, and $f(t,y)$ with respect to $t$ provided by Assumption~\eqref{ass:main} gives us
\begin{equation}\label{eq:cont_adjoint}
    \begin{aligned}
    & \forall \epsilon_p > 0, \exists \delta_p, s.t. \forall t \in [0,T]: | t - \tau| \le \delta_p \Longrightarrow \|p(t) - p(\tau)\|_{H^{*}} \le \epsilon_p,
    \end{aligned}
\end{equation}
\begin{equation}\label{eq:cont_f}
    \begin{aligned}
    & \forall \epsilon_f > 0, \exists \delta_f, s.t. \forall t \in [0,T]: | t - \tau| \le \delta_f \Longrightarrow \|f(t,y) - f(\tau,y)\| \le \epsilon_f, 
    \end{aligned}
\end{equation}
\begin{equation}\label{eq:cont_state}
    \begin{aligned}
    & \forall \epsilon_y > 0, \exists \delta_y, s.t. \forall t \in [0,T]: | t - \tau| \le \delta_y \Longrightarrow \|y(t) - y(\tau)\|_{H} \le \epsilon_y. 
    \end{aligned}
\end{equation}
Denote
\begin{equation}\label{eq:def_epsilon_delta}
    \begin{aligned}
    & \bar{y}(t) = y(t) - y(\tau), \quad \delta = \min\{\delta_p, \delta_f, \delta_y\}, \quad \epsilon = \max\{\epsilon_p, \epsilon_f, \epsilon_y \}.\\
    \end{aligned}
\end{equation}
Furthermore, by Assumption~\ref{ass:main}, for every fixed $t \in [0,T]$, we have the expansion
\begin{equation}\label{eq:Fr_dif}
    \begin{aligned}
    & f(t,y + h) = f(t, y) + K(t) h + o(h),\\
    \end{aligned}
\end{equation}
with $\|h\|_H \rightarrow 0$, $K(t)$: $H \to \mathfrak{L} (U,H)$ being a bounded linear operator. Moreover, denote the operator norms $\bar{P} = \| P^{\star}(\tau)P(\tau) \|_{op}$ and $\bar{f} = \|f(t,y(t))\|_{op}$. 
Combining \eqref{eq:soll2problem} and \eqref{eq:migrep} we receive
\begin{equation}\label{eq:bound_mig_1}
\begin{aligned}
& \frac{dJ_1}{d\lambda^+}(t,u_{\text{opt}}(\tau)) = \la p(t), f(t,y(t))\big( \Pi(\tau)[P^{\star}(\tau)P(\tau)u_1(\tau)+P^{\star}(\tau)\alpha_d] - u_1(t)\big) \ra_{H^{*},\; H}\\
&\quad \quad \quad \quad \quad \quad = \la p(t),f(t,y(t))\Pi(\tau) P^{\star}(\tau)\alpha_d \ra_{H^{*},\; H}\\
&\quad \quad \quad \quad \quad \quad + \la p(t),f(t,y(t))\Pi(\tau)P^{\star}(\tau)P(\tau)u_1(\tau) \ra_{H^{*},\; H} - \la p(t),f(t,y(t)) u_1(t) \ra_{H^{*},\; H}. \\
\end{aligned}    
\end{equation}
The second part of the right-hand side can be bounded using the norms of the adjoint and the reference control
\begin{equation}\label{eq:bound1}
\begin{aligned}
& \la p(t),f(t,y(t))\Pi(\tau)P^{\star}(\tau)P(\tau)u_1(\tau) \ra_{H^{*},\; H} - \la p(t),f(t,y(t)) u_1(t) \ra_{H^{*},\; H}\\
&\quad \quad \quad \quad \quad \quad \quad \quad \quad \quad \quad \quad \quad \quad \quad \quad \quad \quad \quad \le \|p(t)\|_{H^{*}} \bar{f} \big( \|\Pi(\tau)\|_{op} \bar{P} \|u_1(\tau)\|_U + \|u_1(t)\|_U \big). \\
\end{aligned}    
\end{equation}
The remaining part of \eqref{eq:bound_mig_1} can be rearranged as
\begin{equation}\label{eq:bound_mig_2}
\begin{aligned}
& \la p(t),f(t,y(t))\Pi(\tau) P^{\star}(\tau)\alpha_d \ra_{H^{*},\; H} = \la p(\tau), f(t,y(t)) \Pi(\tau)P^{\star}(\tau)\alpha_d \ra_{H^{*},\; H}\\
&\quad \quad \quad \quad \quad \quad \quad \quad \quad \quad \quad \quad\quad \quad \quad  + \la p(t) - p(\tau),f(t,y(t))\Pi(\tau) P^{\star}(\tau)\alpha_d \ra_{H^{*},\; H}.
\end{aligned}    
\end{equation}
One of the parts can be directly bounded using \eqref{eq:cont_adjoint}
\begin{equation}\label{eq:bound2}
\begin{aligned}
& \la p(t) - p(\tau),f(t,y(t))\Pi(\tau) P^{\star}(\tau)\alpha_d \ra_{H^{*},\; H} \le \epsilon_p \bar{f} \|\Pi(\tau)\|_{op} \|P^{\star}\|_{op} |\alpha_d|.\\
\end{aligned}    
\end{equation}
The remaining part of \eqref{eq:bound_mig_2} can be rewritten using \eqref{eq:Fr_dif}
\begin{equation}\label{eq:bound_mig_3}
\begin{aligned}
& \la p(\tau), f(t,y(t)) \Pi(\tau)P^{\star}(\tau)\alpha_d \ra_{H^{*},\; H} = \la p(\tau), f(t,y(\tau)) \Pi(\tau)P^{\star}(\tau)\alpha_d \ra_{H^{*},\; H}\\
&\quad \quad \quad \quad \quad \quad \quad \quad \quad \quad \quad \quad\quad \quad \quad  + \la p(\tau),K(t) \bar{y}(t) \Pi(\tau) P^{\star}(\tau)\alpha_d \ra_{H^{*},\; H}\\
&\quad \quad \quad \quad \quad \quad \quad \quad \quad \quad \quad \quad\quad \quad \quad  + \la p(\tau), o( \bar{y}(t)) \Pi(\tau) P^{\star}(\tau)\alpha_d \ra_{H^{*},\; H}.\\
\end{aligned}    
\end{equation}
Last two terms of the right-hand side of \eqref{eq:bound_mig_3} can be bounded using \eqref{eq:cont_state}
\begin{equation}\label{eq:bound3}
\begin{aligned}
&\la p(\tau),K(t) \bar{y}(t) \Pi(\tau) P^{\star}(\tau)\alpha_d \ra_{H^{*},\; H} + \la p(\tau), o( \bar{y}(t)) \Pi(\tau) P^{\star}(\tau)\alpha_d \ra_{H^{*},\; H} \le\\ 
&\quad \quad \quad \quad \quad \quad \quad \quad \quad \quad \quad \quad\quad \quad \quad \le \|p(\tau)\|_{H^{*}}  \big( \|K(t)\|_{op} + \epsilon_y \big)\epsilon_y \|\Pi(\tau)\|_{op} \|P^{\star}\|_{op} |\alpha_d|. 
\end{aligned}    
\end{equation}
Finally, the first term of \eqref{eq:bound_mig_3} can be transformed and bounded using \eqref{eq:cont_f}
\begin{equation}\label{eq:bound4}
\begin{aligned}
&\la p(\tau), f(t,y(\tau)) \Pi(\tau)P^{\star}(\tau)\alpha_d \ra_{H^{*},\; H} = \la p(\tau), f(\tau,y(\tau)) \Pi(\tau)P^{\star}(\tau)\alpha_d \ra_{H^{*},\; H}\\
&\quad \quad \quad \quad \quad \quad \quad \quad \quad \quad \quad \quad\quad \quad \quad + \la p(\tau), (f(t, y(\tau)) - f(\tau,y(\tau))) \Pi(\tau)P^{\star}(\tau)\alpha_d \ra_{H^{*},\; H}\\
&\quad \quad \quad \quad \quad \quad \quad \quad \quad \quad \quad \quad\quad \quad \quad \le \|p(\tau)\|_{H^{*}} \epsilon_f \|\Pi(\tau)\|_{op} \|P^{\star}\|_{op} |\alpha_d| \\
&\quad \quad \quad \quad \quad \quad \quad \quad \quad \quad \quad \quad\quad \quad \quad + P(\tau) \Pi(\tau) P^{\star}(\tau)\alpha_d.
\end{aligned}    
\end{equation}
Combining all previously obtained bounds \eqref{eq:bound1},\eqref{eq:bound2},\eqref{eq:bound3} and \eqref{eq:bound4} with $\epsilon$ defined in \eqref{eq:def_epsilon_delta}, we can bound the mode insertion gradient \eqref{eq:mig}
\begin{equation}\label{eq:finalbound}
    \begin{aligned}
    & \frac{dJ_1}{d\lambda^+}(t,u_{\text{opt}}(\tau)) \le  \|p(t)\|_{H^{*}} \bar{f} \big( \|\Pi(\tau)\|_{op} \bar{P} \|u_1(\tau)\|_U + \|u_1(t)\|_U \big) + \|p(\tau)\|_{H^{*}}\|\Pi(\tau)\|_{op} \|P^{\star}\|_{op} \epsilon |\alpha_d|\\
     &\quad \quad \quad \quad \quad \quad + \|p(\tau)\|_{H^{*}} \|\Pi(\tau)\|_{op} \|P^{\star}\|_{op} \big( \|K(t)\|_{op} + \epsilon \big)\epsilon |\alpha_d|  + \bar{f} \|\Pi(\tau)\|_{op} \|P^{\star}\|_{op} \epsilon |\alpha_d| \\
     &\quad \quad \quad \quad \quad \quad + P(\tau) \Pi(\tau) P^{\star}(\tau)\alpha_d.
    \end{aligned}
\end{equation}
Here we can see, that all the terms on the right-hand side of \eqref{eq:finalbound} which contain $|\alpha_d|$ also contain $\epsilon$, which can be made arbitrary small by the appropriate choice of $\delta$ defined in \eqref{eq:def_epsilon_delta}.
This gives us that for all $\tau \in (0,T)$, there exist a neighborhood $N(\tau) \in (0,T)$ and $\alpha_d < 0$ which has to satisfy the following inequality 
\begin{equation}
    \begin{aligned}
    & P(\tau) \Pi(\tau) P^{\star}(\tau)\alpha_d \le  -\|p(t)\|_{H^{*}} \bar{f} \big( \|\Pi(\tau)\|_{op} \bar{P} \|u_1(\tau)\|_U + \|u_1(t)\|_U \big)\\
    &\quad \quad \quad \quad \quad \quad \quad \quad - \|p(\tau)\|_{H^{*}} \|\Pi(\tau)\|_{op} \|P^{\star}\|_{op} \big( \|K(t)\|_{op} + \epsilon \big)\epsilon |\alpha_d|\\
    &\quad \quad \quad \quad \quad \quad \quad \quad - \|p(\tau)\|_{H^{*}}\|\Pi(\tau)\|_{op} \|P^{\star}\|_{op} \epsilon |\alpha_d|\\
    &\quad \quad \quad \quad \quad \quad \quad \quad - \bar{f} \|\Pi(\tau)\|_{op} \|P^{\star}\|_{op} \epsilon |\alpha_d|, \\
    \end{aligned}
\end{equation}
such that
\begin{equation}
    \begin{aligned}
    & \frac{dJ_1}{d\lambda^+}(t,u_{\text{opt}}(\tau)) \le 0
    \end{aligned}
\end{equation}
for all $t \in N(\tau)$. And this concludes the proof.
\end{proof}

The result in Corollary~\ref{col:2.3} gives us, that for every $\tau \in (0,T)$, for an appropriately chosen coefficient $\alpha_d$, there exists some length $\bar{\lambda}$ of application time of the control $u_{\text{opt}}(\tau)$, found by the SAC principle \eqref{eq:l2problem}, such that for all $t \in [\tau - \frac{\bar{\lambda}}{2},\tau + \frac{\bar{\lambda}}{2}]$, $\frac{dJ_1}{d\lambda^+}(t,u_{\text{opt}}(\tau)) \le 0$. And by using \eqref{eq:mig} we receive that for every fixed point $t \in [\tau - \frac{\bar{\lambda}}{2},\tau + \frac{\bar{\lambda}}{2}]$: 
\begin{equation}
    \begin{aligned}
    &J_1(u_{\bar{\lambda},t,u_{\text{opt}}(\tau)}) \le J_1(u_1).
    \end{aligned}
\end{equation}
This means, that the change in the cost can be locally bounded from above, and controlled by the appropriately chosen $\alpha_d$ and $\bar{\lambda}$. In practice the action duration $\bar{\lambda}$ can be found using a standard line search method. We note that, in addition, practical versions of SAC take into account a small computation time $\Delta_c t$ needed to solve \eqref{eq:soll2problem} numerically as well as additional steps such as computation of efficient application times $\tau$. Moreover, saturation techniques for control constraints can be used in each iteration. These steps can be implemented as proposed in \cite{AnsariMurphey2016} and shall therefore not be discussed here further.

\subsection{Linear systems case}\label{subsec:linear}
In this subsection we specialize the main results of the previous subsection to the case $f(t,y) = B$, with $B$ being a bounded control operator on a control space $U$ with images in $H$, that is, we consider now
\begin{equation}\label{eq:forwardproblemlinear}
\dot{y}(t)=Ay(t) + Bu(t),~u(t) \in U,\quad t \in (0,T), \quad y(0) = y_0,
\end{equation}
with the mild solution 
\begin{equation}\label{eq:forwardproblem-mildlinear}
y(t) = e^{tA} y_0+\int_0^t e^{(t-s)A}Bu(s)\,ds,\quad t \in [0,T]
\end{equation}
for any $u \in \Ut$ and any $y_0 \in H$. Moreover, the adjoint problem \eqref{eq:adjointproblem} simplifies to
\begin{equation}\label{eq:adjointproblemlinear}
\begin{aligned}
 \dot{p}(t) &= -A^* p(t) - (l_1)_y(y(t))\\
 p(T) &= m_y(y(T))
\end{aligned}
\end{equation}
with the corresponding mild solution 
\begin{equation}\label{eq:adjointproblem-mildlinear}
p(t) = e^{(T-t)A^*} p(T)+\int_t^T e^{(s - t)A^*}(l_1)_y (y(s))\,ds,\quad t \in [0,T]. 
\end{equation} 
Using Theorem~\ref{thm:main}, we have that the problem \eqref{eq:l2problem} has an explicit solution given by
\begin{equation}\label{eq:soll2problemlinear}
 u_{\text{opt}}(\tau) = \Pi(\tau)[P^{\star}(\tau)P(\tau)u_1(\tau)+P^{\star}(\tau)\alpha_d]
\end{equation}
for all $\tau \in [0,T]$, with $P \in C([0,T];U^*)$ defined by
\begin{equation}\label{eq:Pdeflinear}
 P(t)\: w \mapsto \la B^*p(t),w\ra_{U^{*},\; U},\quad t \in [0,T].
\end{equation}

\begin{remark}\label{rem:remark}
It is convenient to identify the Hilbert spaces $H^{*}$ and $U^{*}$ with $H$ and $U$, respectively. We identify $p \in C(0,T;H^{*})$ with $\pi \in C(0,T;H)$ so that
\begin{equation*}
	\langle p(t), v\rangle_{H^{*},\; H} = \langle \pi(t), v\rangle_H, \quad \forall t \in [0,T], \forall v \in H.
\end{equation*}
Applying this to the adjoint problem \eqref{eq:adjointproblem-mildlinear}, we receive the following problem in $H$
\begin{equation}\label{eq:Adjoint_Riesz}
\begin{aligned}
 \dot{\pi}(t) &= -A^{\star} \pi(t) - (\tilde{l}_1)_y(y(t))\\
 \pi(T) &= \tilde{m}_y(y(T)),
\end{aligned} 
\end{equation}
where $(\tilde{l}_1)_y(y(t))$ and $\tilde{m}_y(y(T))$ are unique Riesz representations of $(l_1)_y(y(t))$ and $m_y(y(T))$ in H, respectively, and $A^{\star}$ is the Hilbert adjoint of the unbounded operator $A$.
\end{remark}
For simplicity of the notations, we continue writing $p$ for $\pi$ in Sections \ref{sec:linear-quad} and \ref{sec:numerical}. 
\section{The SAC feedback for linear quadratic problems}\label{sec:linear-quad}

In this section we consider the important special case of quadratic stage costs. In this setting, one can derive a linear feedback law that describes 
in first order the dynamics of the closed loop system when controls given by the SAC-principle are continuously applied as the stage problem is shifted from $[0,T]$ to $[t,t+T]$, with $t>0$. 

\subsection{The SAC feedback for quadratic stage costs}\label{subsec:qcost}
Our analysis concerns 
the case of quadratic stage costs of the form
\begin{equation}\label{eq:J1_quadratic}
 \tilde J_1 = \frac12 \int_t^{t+T} \|Q(y(\tau)-y_d)\|_H^2\,d\tau + \frac12\la P_{T}(y(t+T)-y_d,y(t+T)-y_d\ra_H
\end{equation}
subject to \eqref{eq:forwardproblemlinear} with $y_d=0$, $u_1=0$, $Q$ being a bounded linear operator on $H$ and $P_{T}$ a positive and self-adjoint operator on $H$. In the following, we identify the Hilbert spaces $H^{*}$ and $U^{*}$ with $H$ and $U$, respectively (as described in the Remark~\ref{rem:remark}).

\begin{lemma} Let $F(\tau)\colon H \to H$ be a self-adjoint, positive, bounded linear operator  and the mild solution of the differential Lyapunov equation
\begin{equation}\label{eq:ricatti_eq}
\begin{aligned}
 &\dot{F}(\tau) = -A^{\star}F(\tau) - F(\tau)A-Q^{\star}Q,\quad \tau \in (0,T)~\text{a.e.}\\
 &F(T)=P_{T}.
\end{aligned}
\end{equation}
Then for quadratic costs of the form \eqref{eq:J1_quadratic}, it holds
\begin{equation}\label{eq:explicit_p}
 p(\tau) = F(\tau - t)y(\tau),\quad \tau \in [t,t+T].
\end{equation}
\end{lemma}
\begin{proof}
For \eqref{eq:J1_quadratic} $m_y(y(t+T)) = P_{T}y(t+T)$ and $(l_1)_y(y(\tau)) = Q^{\star}Qy(\tau)$. By \eqref{eq:adjointproblemlinear} and \eqref{eq:ricatti_eq} it holds that $p(t+T) = P_{T} y(t+T) = F(T)y(t+T)=F(t+T-t)y(t+T)$. Furthermore, for any $y \in L^2(t,t+T;H)$, a mild solution of \eqref{eq:ricatti_eq} is given by
\begin{equation}\label{eq:Ricatti_mild}
 F(\tau - t)y = \int_{\tau - t}^{T} e^{(\sigma-\tau + t)A^{\star}}Q^{\star}Qe^{(\sigma - \tau + t)A}y\,d\sigma+e^{(T-\tau + t )A^{\star}}P_{T}e^{(T-\tau + T)A}y,
\end{equation}
see \cite[Part IV, Section~2]{Bensoussan2007}. Now, rewriting \eqref{eq:adjointproblem-mildlinear} using \eqref{eq:J1_quadratic} and \eqref{eq:forwardproblem-mildlinear} with control $u_1 = 0$, we obtain
\begin{equation}\label{eq:p_mild_ric}
 p(\tau) = e^{(t+T-\tau)A^{\star}} P_{T} e^{(t+T)A} y_t + \int_\tau^{t+T} e^{(s - \tau)A^{\star}}Q^{\star}Q e^{sA} y_t\,ds,\quad \tau \in [t,t+T] 
\end{equation}

\begin{equation}\label{eq:Ricatti_mild2}
	\begin{aligned}
 	F(\tau - t)y(\tau) = \int_{\tau - t}^{T} e^{(\sigma-\tau+t)A^{\star}}Q^{\star}Qe^{(\sigma-\tau+t)A}e^{\tau A}y_t\,d\sigma+e^{(T-\tau+t)A^{\star}}P_{T}e^{(T-\tau+t)A}e^{\tau A}y_t \\
 	 = \int_{\tau}^{t+T} e^{(s - \tau)A^{\star}}Q^{\star}Q e^{s A} y_t\,d\sigma + e^{(T-\tau+t)A^{\star}}P_{T}e^{(T+t)A}y_t,\quad 	\tau \in [t,t + T] .
	\end{aligned}
\end{equation}
Equality of the right-hand-sides of \eqref{eq:p_mild_ric} and \eqref{eq:Ricatti_mild2} along with the initial conditions gives us \eqref{eq:explicit_p}.
\end{proof}

Using the Lemma above and taking $p(\tau)$ at point $t$, with $\overline{F} = F(0)$ we have

$$ p(t) = \overline{F}y(t) = e^{T A^{\star}} P_{T} e^{TA} y(t) + \int_t^{t+T} e^{(s - t)A^{\star}}Q^{\star}Q e^{(s-t)A} y(t)\,ds $$

$$ = e^{T A^{\star}} P_{T} e^{TA} y(t) + \int_0^{T} e^{sA^{\star}}Q^{\star}Q e^{sA} y(t)\,ds.$$
Furthermore, incorporating $u_{\text{opt}}$ given by \eqref{eq:soll2problemlinear} into \eqref{eq:forwardproblemlinear} we receive
\begin{equation}\label{eq:forwardproblemlinearclosed}
\begin{aligned}
& \dot{y}(t)=Ay(t) + G(y),\quad t \in (0,T), \quad y(0) = y_0\\
& G(y) = B [(P^{\star}(t)P(t)+R^{\star})(y)]^{-1}B^{\star} \bar{F}y(t)\alpha_d, \quad G: H \rightarrow H.
\end{aligned}
\end{equation}

\begin{theorem}\label{eq:corSAC1storder}
For \eqref{eq:J1_quadratic}, the following statements hold.

\begin{enumerate}
\item The SAC action satisfies
\begin{equation}\label{eq:SAC1storder}
u_{\text{opt}}(t)=\alpha_d(R^{\star})^{-1}B^{\star}\overline{F}y(t)+o(\|y\|_H)~				\quad \text{for}~\|y\|_H \to 0.
\end{equation}
\item The nonlinear operator G from \eqref{eq:forwardproblemlinearclosed} is Fr\'{e}chet differentiable.
\item A linearization at the equilibrium of a continuous application of controls computed by SAC is a system under linear feedback
\begin{equation}\label{eq:closedloop1storder}
\begin{aligned}
 &\dot{y}=\left(A+\alpha_dB(R^{\star})^{-1}B^{\star}\overline{F}\right) y(t),\quad t \geq 0,\\
 &y(0)=y_0.
\end{aligned} 
\end{equation}

\end{enumerate}

\end{theorem}
\begin{proof}

\begin{enumerate}
\item From \eqref{eq:Pdef}, we have that for all $v,w \in U$
\begin{equation*}
 (P^{\star}(t)P(t)v)w = \la P^{\star}(t)P(t)v,w\ra_U = \la P(t)v, P(t) w\ra_{\mathbb{R}} =\la p(t),Bv\ra_H \la p(t),Bw\ra_H,
\end{equation*}
and 
\begin{equation*}
 \la p(t),Bv\ra_H = \la \overline{F}y(t),Bv\ra_H \leq \|\overline{F} y(t)\|_H \|B\|_{op} \|v\|_U.
\end{equation*}
Hence, the self-adjoint operator $P^{\star}(t)P(t)$ depends quadratically on $y$.
Moreover, using that $F$ as a solution of \eqref{eq:ricatti_eq} is independent of $y$, we have
\begin{equation*}
 \|\overline{F}y(t)\|_H \|B\|_{op} \leq C_{B^{\star}}\|\overline{F}y(t)\|_H \leq C_{B^{\star}}C_{F}\|y(t)\|_H, \quad y \in L^2(0,T;H)
\end{equation*}
for some constants $C_{B^{\star}},C_F>0$. The inverse operator \eqref{eq:soll2problem} can be written as
\begin{equation*}
(P^{\star}(t)P(t)+R^{\star})^{-1} = (R^{\star})^{-1} - (R^{\star})^{-1}P^{\star}(t)P(t)(P^{\star}(t)P(t)+R^{\star})^{-1}.
\end{equation*}
With the above, we get from \eqref{eq:soll2problemlinear}
\begin{equation}\label{eq:firstorderest}
\begin{aligned}
 u_{\text{opt}}(t) = (P^{\star}(t)P(t)+R^{\star})^{-1}P^{\star}(t)\alpha_d = (R^{\star})^{-1}P^{\star}(t)\alpha_d+o(\|y\|_H)
\end{aligned}
\end{equation}
for $\|y\|_H$ sufficiently small. Using \eqref{eq:Pdeflinear} in \eqref{eq:firstorderest} yields \eqref{eq:SAC1storder}.

\item Now, we consider the nonlinear part of \eqref{eq:forwardproblemlinearclosed}, $G(y) = B [(P^{\star}(t)P(t)+R^{\star})(y)]^{-1}B^{\star} \bar{F}y(t)\alpha_d$. The operator $(P^{\star}(t)P(t)+R^{\star})(y): u \in U \mapsto \alpha_d \la B^{\star}\bar{F}y, u \ra_U B^{\star} \bar{F}y + R^{\star}u \in H$ is Fr\'{e}chet differentiable with respect to $u \in U$, furthermore, the operator has a bounded inverse (see Th.~\ref{thm:main}). Thus, using the Inverse Function Theorem (see \cite{Invers_diff}), the inverse operator $[(P^{\star}(t)P(t)+R^{\star})(y)]^{-1}$ is Fr\'{e}chet differentiable as an operator on $u$. 

Furthermore, for any fixed $u \in U$, $(P^{\star}(t)P(t)+R^{\star})(y)u = \alpha_d \la B^{\star}\bar{F}y, u \ra_U B^{\star} \bar{F}y + R^{\star}u$ is Fr\'{e}chet differentiable as a function of $y$. So due to the invertibility of the operator $(P^{\star}(t)P(t)+R^{\star})(y)$, using \cite[Theorem 2]{Potthast} we get that $[(P^{\star}(t)P(t)+R^{\star})(y)]^{-1}u$ is Fr\'{e}chet differentiable as a function of $y$. 

Combining these results and the chain rule, we obtain that the operator $G(y) = B [(P^{\star}(t)P(t)+R^{\star})(y)]^{-1}B^{\star} \bar{F}y(t)\alpha_d$ is Fr\'{e}chet differentiable as a function of $y$. 

\item Using the result from the first statement, by using \eqref{eq:firstorderest} in \eqref{eq:forwardproblemlinear}, we directly get a linearization at the equilibrium of a continuous application of controls computed by SAC, which is a system under linear feedback
\begin{equation*}
\begin{aligned}
 &\dot{y}=\left(A+\alpha_dB(R^{\star})^{-1}B^{\star}\overline{F}\right) y(t),\quad t \geq 0,\\
 &y(0)=y_0.
\end{aligned} 
\end{equation*}

\end{enumerate}
\end{proof}


For given $A, B, Q, P_T,$ and $R$, the closed-loop system \eqref{eq:closedloop1storder} can be analyzed for asymptotic and exponential stability. For instance, it is well known that if under the above settings, the operator $A + DG(0)$, which is the infinitesimal generator of a $C_0$-semigroup $T(t)$, satisfies
$$ \lim_{t \rightarrow \infty} t^{-1} \log \|T(t)\| < 0,$$
then $y_d = 0$ is exponentially asymptotically stable.   See  \cite[Corollary 2.2.]{Kato}. In the parabolic case with asymptotic stability of the linearization, we can use results from \cite[Theorem 5.1.1]{Henry} to get asymptotic stability of the equilibrium of the original system in the appropriate space. We do this prototypically for a selected example concerning an unstable parabolic problem in the next subsection. Moreover, for reaction-diffusion equations and for wave equations, linearization and stability analysis can be done using the appropriate generalizations of the Hartman-Grobman theorem (see e.g., \cite{Lu1991} and \cite{RodriguesRuasFilho1997}, \cite{HG2016}, respectively).

\subsection{Unstable heat equation}\label{subsec:heateq}
In this subsection, we consider the stabilization of the one-dimensional reaction-diffusion process
\begin{equation}\label{eq:1dproblem}
\begin{aligned}
& y_t(t,x) = y_{xx}(t,x) + \mu y(t,x) + \sqrt{\beta}\chi_{(a,b)}(x) u(t,x) , \quad \text{on} \quad \Omega_t := [0,\infty) \times (0,L),\\
& y(t,0)=y(t,L)=0, \quad y(0,x)=y_0(x)
\end{aligned}
\end{equation}
at $y_d=0$ with the quadratic stage costs 
\begin{equation}\label{eq:1dcost}
 \tilde J_1 = \frac12 \int_0^T \int_0^L (\bar{q} (y(t,x)-y_d(x)))^2\,dx\,dt
\end{equation}
for real constants $\beta,\bar{q}>0$, $\chi_{(a,b)}(x) = 1$, for $x\in (a,b)$ and $0$ otherwise, $\mu > \left(\frac{\pi}{L}\right)^2$, $(a,b) \subset (0,L)$ with the SAC framework of Section~\ref{sec:framework}. Here we take $\mu$ larger then the smallest eigenvalue of $(-\Delta)$. We note that the solution of \eqref{eq:1dproblem} without control, i.e., $u(t)=u_1(t)=0$, is exponentially unstable.
The stabilization of \eqref{eq:1dproblem} with classical MPC schemes was investigated in \cite{AltmuellerGruene2012}. 

With the spaces $H = L^2(0,L)$, $U = L^2(0,L)$, the operators $A$ and $B$ defined by $Ay=\mu y + \Delta y$ for $y \in D(A) = H_0^1(0,L)\cap H^2(0,L)$ and $Bu=\sqrt{\beta}u$ for $u \in U$,
the control problem \eqref{eq:1dproblem} can be written as 
\begin{equation*}
\dot{y}(t)=Ay(t) + Bu(t),~t \in (0,T),\quad y(0) = y_0,
\end{equation*}
see, e.g., \cite{Bensoussan2007}. With $P_T=0$, $Qy=\bar{q}y$, the cost function \eqref{eq:1dcost} has the quadratic form \eqref{eq:J1_quadratic}. Moreover, for the SAC principle, 
we choose $R$ in \eqref{eq:l2problem} as the identity in $H$. Then, the closed-loop system \eqref{eq:closedloop1storder} becomes
\begin{equation}\label{eq:1dcase}
\begin{aligned}
 & y_t = Ay + \alpha_d \beta \overline{F}_q y, \quad \text{on} \quad \Omega_t := [0,\infty) \times (0,L), \quad y(t,0)=y(t,L)=0, \quad y(0,x)=y_0(x)\\
\end{aligned}
\end{equation}
with $\overline{F}_{q}$ as a solution of the Riccati equation \eqref{eq:ricatti_eq} given by
\begin{equation}\label{eq:defFbarq}
 \overline{F}_{q}y(t)=\int_0^T e^{\sigma A^{\star}}Q^{\star}Qe^{\sigma A}y(t)\,d\sigma=\int_0^T \bar{q}^2\left(e^{\sigma A}\right)^2 y(t)\,d\sigma,
\end{equation}
from \eqref{eq:Ricatti_mild} and using that $A$ is self-adjoint.

In the following we analyze the closed-loop system \eqref{eq:1dcase} with control $u$ on the full space $(a,b) = (0,L)$. We provide computational studies in Section~\ref{sec:numerical} for the control on both full and subdomain. First, we characterize solutions of the closed-loop system \eqref{eq:1dcase} using a product approach.
\begin{lemma}\label{lem:solcharac}
Let $\phi_k$, $k=1,...,\infty$, be eigenfunctions, $D_k$ being eigenvalues of the Dirichlet-Laplace operator $\Delta y$ on $D(A) \subset H$ and assume that $\mu \neq -D_k$ for all $k \in \mathbb{N}$. Furthermore, let 
$$\alpha_k(t) = \chi_k e^{\left((\mu + D_k) - \frac{\alpha_d \beta \bar{q}^{2}}{2(\mu + D_k)} + \frac{\alpha_d \beta \bar{q}^{2}}{2(\mu + D_k)} e^{2 T (\mu + D_k)} \right)t }$$ 
with $\alpha_k(0) = \chi_k$ and constants $\chi_k := \langle y_0, \phi_k \rangle_{L^2(0,L)}$, $k=1,...,\infty$. Then, the solution $y$ of \eqref{eq:1dcase} is within the set of functions 
\begin{equation}\label{prod_approach}
\begin{aligned}
 & y(t,x) = \sum_{k=1}^{\infty} \alpha_k \phi_k, \quad x \in (0,L),~ t \in [0,\infty), \\
\end{aligned}
\end{equation} 
for which the sum converges.
\end{lemma}
\begin{proof}
In the operator form the semigroup $e^{T A}$ acts on the function $\phi_k(\cdot)$ as $e^{T (\mu + D_k)} \phi_k (\cdot)$, see, e.g., \cite{EngelNagel2000}. Hence, using \eqref{eq:defFbarq} and the ansatz \eqref{prod_approach}, 
we obtain from \eqref{eq:1dcase} the equation
\begin{equation}\label{explicit_product_1}
\begin{aligned}
 & \dot{\alpha}_k (t) \phi_k(x) = \alpha_k (t) \phi_{k}'' (x) + \mu \alpha_k (t) \phi_k (x) + \alpha_d \beta \int_0^{T} \alpha_k (t) \bar{q}^{2} e^{2 \sigma (\mu + D_k)} \phi_k (x) d\sigma.
\end{aligned}
\end{equation} 
Here, $\phi_k$ represents the spectral decomposition of $\phi$ in $L^2 (0,L)$. Now dividing \eqref{explicit_product_1} by $\phi_k(x)$ and substituting $\frac{\phi_{k}''(x)}{\phi_k(x)}$ by $D_k$ we 
receive a first order ODE 
\begin{equation}\label{ord_dif_eq}
\dot{\alpha}_k (t) = (\mu + D_k)\alpha_k(t) + \alpha_d \beta \bar{q}^{2} \alpha_k (t) \int_0^{T} e^{2\sigma(\mu + D_k)} d\sigma.
\end{equation} 
By computing the integral explicitly and rearranging terms, we get
\begin{equation*}
\begin{aligned}
\dot{\alpha}_k (t) 
&=\alpha_k(t)\left( \left(\mu + D_k \right) + \frac{\alpha_d \beta \bar{q}^{2}}{2(\mu + D_k)} e^{2T (\mu + D_k)} - \frac{\alpha_d \beta \bar{q}^{2}}{2(\mu + D_k)} \right).
\end{aligned}
\end{equation*} 
Now define $\delta_k = (\mu + D_k)$ and $\kappa_k = \frac{\alpha_d \beta \bar{q}^{2}}{2(\mu + D_k)}$. With this, we obtain 
\begin{equation}\label{ord_dif_eq_3}
\dot{\alpha}_k (t) = \alpha_k(t)\left( \delta_k + \kappa_k e^{2T\delta_k} - \kappa_k \right).
\end{equation} 
Using the initial value $\alpha_k(0) = \chi_k$ we can solve \eqref{ord_dif_eq_3}:
\begin{equation*}
\int_{\chi_k}^{\alpha_k(t)} \frac{d\alpha_k}{\alpha_k} = \int_0^t \left( \delta_k + \kappa_k e^{2T\delta_k} - \kappa_k \right) ds,
\end{equation*}
\begin{equation*}
\ln (\alpha_k(t)) = \ln (\chi_k) + \left(\delta_k + \kappa_k e^{2T\delta_k} - \kappa_k \right)t.
\end{equation*}
So the solution of the equation \eqref{ord_dif_eq} is given by
\begin{equation*}
\alpha_k(t) = \chi_k e^{\left(\delta_k + \kappa_k e^{2T\delta_k} - \kappa_k \right)t}.
\end{equation*}
Using that $\phi_k$, $k=1,\ldots,\infty$, is a basis of $H$ concludes the proof. 
\end{proof}

The following result concerns the asymptotic stability of the closed-loop system \eqref{eq:1dcase} for $t \to \infty$ in dependency of the most important parameters $T,\bar{q},\beta$ and $\alpha_d$
of the SAC-principle.

\begin{theorem}\label{thm:stabclosedloopfirstorder}
Let $D_k$ being the eigenvalues of the Dirichlet-Laplace operator $\Delta y$ on $D(A) \subset H$ and assume that $\mu \neq -D_k$ for all $k \in \mathbb{N}$. Then for any $T, \bar{q}, \beta > 0$, there exists $\bar{\alpha}_d < 0$ such that the closed-loop system \eqref{eq:1dcase} is asymptotically stable in $y_d=0$ for any $\alpha_d \leq \bar{\alpha}_d$.
\end{theorem}
\begin{proof}
We use the characterization of solutions to the closed-loop system~\eqref{eq:1dcase} from \eqref{lem:solcharac}. From $\|y(t,x)\|_H \leq \sum_{k=1}^\infty |\alpha_k(t)|\|\phi_k\|_H$, we want to show that $|\alpha_k(t)| \to 0$ uniformly in $k$ for $t \to \infty$. To this end, and noting that, for all $k=1,\ldots,\infty$, the coefficients $\alpha_k$ are continuously differentiable for all $t$, we take a look on the derivative of $\alpha_k$
\begin{equation*}
\begin{aligned}
& \dot{\alpha}_k (t) = \chi_k \left((\mu + D_k) - \frac{\alpha_d \beta \bar{q}^{2}}{2(\mu + D_k)}  + \frac{\alpha_d \beta \bar{q}^{2}}{2(\mu + D_k)} e^{2 T (\mu + D_k)} \right) e^{\left((\mu + D_k) - \frac{\alpha_d \beta \bar{q}^{2}}{2(\mu + D_k)}  + \frac{\alpha_d \beta \bar{q}^{2}}{2(\mu + D_k)} e^{2 T (\mu + D_k)} \right) t }\\
& \dot{\alpha}_k (t) = \alpha_k (t) \left((\mu + D_k) - \frac{\alpha_d \beta \bar{q}^{2}}{2(\mu + D_k)}  + \frac{\alpha_d \beta \bar{q}^{2}}{2(\mu + D_k)} e^{2 T (\mu + D_k)} \right).
\end{aligned}
\end{equation*}
The term $\frac{\dot{\alpha}_k(t)}{\alpha_k(t)} $ must be less than zero, so that the function $|\alpha_k(\cdot)|$ decreases close to zero. 
To guarantee this we require
\begin{equation}\label{stab_cond}
\left((\mu + D_k) - \frac{\alpha_d \beta \bar{q}^{2}}{2(\mu + D_k)}  + \frac{\alpha_d \beta \bar{q}^{2}}{2(\mu + D_k)} e^{2 T (\mu + D_k)} \right) \leq C < 0,
\end{equation}  
where $C = - \min\{|\mu + D_k| : k \in \mathbb{N}\}$ is independent of $k$.

Consider we first the case when $(\mu + D_k) < 0$. Then reorganizing \eqref{stab_cond} and multiplying both sides by $2(\mu + D_k)$ we obtain
\begin{equation*}
\alpha_d \beta \bar{q}^{2}\left(e^{2 T(\mu + D_k)} - 1 \right) > -2(\mu + D_k)^2 + C(\mu + D_k).
\end{equation*}
From $(\mu + D_k) < 0$, we have that $\left(e^{2 T(\mu + D_k)} - 1 \right) < 0$ and obtain the condition
\begin{equation*}
\alpha_d \beta \bar{q}^{2} < \frac{-2(\mu + D_k)^2 + C(\mu + D_k)}{\left(e^{2 T(\mu + D_k)} - 1 \right)}.
\end{equation*}
This inequality holds for any $\alpha_d < 0$. 

Now we consider the second case when $\mu + D_k > 0$. In order to have under this condition the inequality $\left(e^{2 T(\mu + D_k)} - 1 \right) > 0$, 
we obtain from \eqref{stab_cond} an explicit inequality constraint for $\alpha_d$ for different $k$ 
\begin{equation}\label{stab_cond_delta_less_0_fin_new}
\alpha_{d,k} < \frac{-2(\mu + D_k)^2 + C(\mu + D_k)}{\beta \bar{q}^{2} \left(e^{2 T(\mu + D_k)} - 1 \right)}.
\end{equation}
Hence, for each $k=1,\ldots,\infty$ we can find appropriate $\alpha_{d,k}$, under conditions \eqref{stab_cond_delta_less_0_fin_new}, for which asymptotic stability holds. 

Since there are only finitely many $(\mu + D_k) > 0$, we can always find an $\alpha_d$ which implies \eqref{stab_cond} for all $k=1,\ldots,\infty$ and for which $\|y(t)\|_{H}$ is decreasing. 
For this we can for example choose
\begin{equation*}
\begin{aligned}
\bar{\alpha}_d = \min\{\alpha_{d,k} : (\mu + D_k) > 0,~k \in \mathbb{N}\}. 
\end{aligned}
\end{equation*} 
This completes the proof.
\end{proof}

Theorem~\ref{thm:stabclosedloopfirstorder} together with \cite[Theorem 5.1.1]{Henry}, in which we have for our case $X = L^2(0, L)$ and $X^{\alpha} = X^{\frac{1}{2}} = H^1_0 (0, L)$, gives us that the closed-loop system obtained after implementation of SAC is locally uniformly asymptotically stable at equilibrium in $X^{\alpha} = H^1_0 (0, L)$.

Together with the results of Subsection~\ref{subsec:qcost}, this yields that continuously applied SAC with any $T,\bar{q}>0$ and a sufficiently small $\alpha_d<0$ stabilizes \eqref{eq:1dproblem} in $y_d=0$ if $\|y_0\|$ is sufficiently small. Our numerical experiments indeed reveal that this result does not extend to global asymptotic stability, i.e., in general, one cannot find a fixed $\alpha_d$ for which stability of the closed-loop is guaranteed for any $y_0 \in Y$. Moreover, too small $\alpha_d$ result in very large control actions, so that a sufficiently short time stepping is needed for a numerical realization in order to avoid overshooting. This suggests choosing $\alpha_d$ depending on $y_0$, for example by setting $\alpha_d = \gamma J_1(u_1)$ with some $\gamma<0$.

If we take $\mu = \left(\frac{\pi}{L}\right)^2$ or smaller, then any negative constant $\gamma$ or even small enough (by absolute value) negative $\alpha_d$ would lead to stabilization of the state. This indicates that SAC actions also qualify for rapid stabilization. The corresponding analysis on stabilization rates may be considered in future work.

In Subsection~\ref{subsec:numerical} we provide a parameter study for the choice of the constants $\gamma$ and $T$.
\section{Discretization and numerical results}\label{sec:numerical}
In this section we will provide some theoretical results concerning Galerkin approximations of SAC actions for linear PDEs of parabolic type and nonlinear costs along with the numerical results for the unstable heat equation under certain conditions. In particular, we investigate numerically the stabilization properties on SAC for control on a subdomain and the case of partial observations. Moreover, we see how disturbance in the instability constant will affect the results, and make a comparison of SAC with standard LQR method.  
\subsection{Galerkin approximations}\label{subsec:galerkin}
In order to obtain a finite-dimensional controller, we consider here Galerkin approximations. On the level of the discretizations, we can then compare the proposed late-lumping control actions 
with those of \cite{AnsariMurphey2016}  applied to a finite-dimensional approximation of the PDE. 
For this subsection, we identify the Hilbert spaces $H^{*}$ and $U^{*}$ with $H$ and $U$, respectively (as described in the Remark~\ref{rem:remark}). Furthermore, we will work under the assumptions that $A$ is induced by a bilinear form in the following settings: 
\begin{enumerate}
	\item $V \hookrightarrow H \simeq H^{*} \hookrightarrow V^{*}$ is a Gelfand triple, $V$ separable Hilbert space. 
	\item $a(\cdot, \cdot): V \times V \rightarrow \mathbb{R}$ is a bilinear form and there are $\alpha, \beta > 0$ and $\gamma \ge 0$ with 
	\begin{equation*}
			\begin{aligned}
				& |a(v,w)| \le \alpha \|v\|_V \|w\|_V \quad \forall v, w \in V, \\
				& a(v,v) + \gamma \|v\|_H^2 \ge \beta \|v\|_V^2 \quad \forall v \in V. 
			\end{aligned}
	\end{equation*}
\end{enumerate}
	 
Under these settings and the Assumption~\ref{ass:main}, the corresponding operator $A\colon D(A) \to H$, densely defined by 
$$\la Ay, k \ra_{V^{*},\; V} = a(y,k) \quad \forall y \in D(A) := \{u \in V; v \mapsto a(u,v) \text{ is continuous w.r.t. } \|\cdot \|_H\},~ k \in V,$$ generates a $C_0$-semigroup (see \cite[Theorem~3, p.330]{Dautray_Lions}).

The mild solution $y \in C(0,T;H)$ of \eqref{eq:forwardproblemlinear} and the mild solution $\pi \in C(0,T;H)$ of \eqref{eq:Adjoint_Riesz} then coincides with the weak solutions given by $\tilde{y} \in W(0,T;H, V)$ and $\tilde{p} \in W(0,T;H, V)$, respectively, for almost every $t \in [0,T]$
\begin{equation}\label{eq:mild_forward}
\la \tilde{y}_t(t), k \ra_{V^*, \; V} = \la A\tilde{y}, k \ra_{V^*,\; V} +\la Bu(t), k\ra_{V^*, \; V},\quad \tilde{y}(0)=\tilde{y}_0,\quad \text{for all}~k\in V,
\end{equation}
\begin{equation}\label{eq:mild_adjoint}
\la \tilde{p}_t(t), k \ra_{V^*, \; V} = -\la A^{\star}\tilde{p}, k \ra_{V^*,\; V} -\la (l_1)_{\tilde{y}}(\tilde{y}(t)), k\ra_{V^*, \; V},\quad \tilde{p}(T) = m_{\tilde{y}}(\tilde{y}
(T)),\quad \text{for all}~k\in V,
\end{equation}
see \cite{Ball1977}, \cite{Bensoussan2007}.

Using Remark~\ref{rem:remark}, we also have for almost every $t \in [0,T]$
\begin{equation}\label{eq:riesz_conn}
	\langle p(t), v\rangle_{H^{*},\; H} = \langle \tilde{p}(t), v\rangle_H, \quad \text{for all}~v \in H,
\end{equation}
with $p\in C(0,T;H^*)$ being the mild solution of \eqref{eq:adjointproblemlinear}.

Let $V_h^1 \subset V$ be a finite-dimensional subspace of $V$ with a basis $(\phi_i)_{i=1}^N$, $V^2_h \subset V$ be a finite-dimensional subspace of $V$ with a basis $(\kappa_i)_{i=1}^K$ and $U_h \subset U$ be a finite-dimensional subspace of $U$ with basis $(\psi_i)_{i=1}^M$. With the ansatz
\begin{equation}\label{eq:discrete-Ansatz}
\tilde{y}(t)=\sum_{i=1}^N \hat{y}_i(t)\phi_i,\quad \tilde{p}(t)=\sum_{i=1}^K \hat{p}_i(t)\kappa_i,\quad u(t)=\sum_{i=1}^M \hat{u}_i(t)\psi_i,
\end{equation}
we get an approximation of \eqref{eq:mild_forward} by
\begin{equation}\label{eq:forward_discrete}
  \mathcal{M} \dot{\hat{y}}(t)=\mathcal{A} \hat{y}(t)+\mathcal{B} \hat{u}(t),
  \quad \mathcal{M} \hat{y}(0)=(\la y_0,\phi_1\ra_H,\ldots,\la y_0,\phi_N\ra_H)^\top
\end{equation}
with matrices 
\begin{equation*}
\mathcal{M}=(\la\phi_j,\phi_i\ra_H)_{i,j=1,\ldots,N},~\mathcal{A}=( a(\phi_j,\phi_i))_{i,j=1,\ldots,N},~\text{and}~\mathcal{B}=(\la B\psi_j,\phi_i\ra_H)_{i=1,\ldots,N,~j=1,\ldots,M}, 
\end{equation*}
see, e.g., \cite{Hinze_Pinnau}. 

Similar, with $\hat{l}_1\colon \RR^N \to \RR$ and $\hat{m}\colon \RR^N \to \RR$ defined as
\begin{equation*}
\hat{l}_1(\hat{y}) \mapsto l_1\left(\sum_{k=1}^N \hat{y}_k \phi_k\right),
\quad \hat{m}(\hat{y}) \mapsto m\left(\sum_{k=1}^N \hat{y}_k \phi_k\right)
\end{equation*}
we get an approximation of \eqref{eq:stageproblem} by
\begin{equation}\label{eq:J_1_discrete}
 \hat{J}_1(\hat{y})=\int_0^T \hat{l}_1(\hat{y}(t))\,dt+\hat{m}(\hat{y}(T)),
\end{equation}
and an approximation of \eqref{eq:mild_adjoint} by
\begin{equation}\label{eq:adjoint_discrete}
\tilde{M} \dot{\hat{p}}(t) = -\tilde{A} \hat{p}(t)-\tilde{l}_1(\hat{y}(t)), \quad \tilde{M} \hat{p}(T)=\tilde{m}(\hat{y}(T))
\end{equation}
with $\tilde{M}=(\la\kappa_j,\phi_i\ra_H)_{i=1\ldots,N,~j=1,\ldots,K}$, $\tilde{A}=( a(\phi_j, \kappa_i))_{j=1,\ldots,N,~ i=1,\ldots,K}$,
and
\begin{equation*}
\tilde{l}_1(\hat{y}(t))=(\la (\hat{l}_1)_{\hat{y}}(\hat{y}(t)),\phi_j\ra_H)_{j=1,\ldots,N},\quad \tilde{m}(\hat{y}(t))=(\la \hat{m}_{\hat{y}}(\hat{y}(T)),\phi_j\ra_H)_{j=1,\ldots,N}. 
\end{equation*}
This yields the following late-lumping SAC control action.

\begin{proposition}\label{prop:optimizediscretize} For a discretized reference control $\hat{u}_1$ such that $u_1(\tau)=\sum_{i=1}^M \hat{u}_{1,i}(\tau) \psi_i$, $\hat{p}$ being a solution of \eqref{eq:adjoint_discrete},  
$\tilde{\Lambda}(\tau)=\tilde{B} \hat{p}(\tau)\hat{p}^\top (\tau) \tilde{B}^\top$ with 
$\tilde{B}=(\la \kappa_j,B\psi_i\ra_H)_{i=1,\ldots,M,~j=1,\ldots,K}$, 
and $\tilde{R} = (\la \psi_i, R\psi_j \ra_U)_{i,j=1,\ldots,M}$ an approximation of \eqref{eq:soll2problemlinear} is given by 
\begin{equation}\label{eq:soll2problem_discrete}
\hat{u}_{\text{opt}}(\tau)=\left(\tilde{\Lambda}+\tilde{R}^\top\right)^{-1}\left[ \tilde{\Lambda}\hat{u}_1(\tau)+\alpha_d\tilde{B} \hat{p}(\tau)\right],
\end{equation}
for all $\tau \in [0,T]$.
\end{proposition}
\begin{proof}
Under the given spaces identifications, Remark~\ref{rem:remark}, and \eqref{eq:riesz_conn} we can rewrite \eqref{eq:soll2problemlinear} in the following form
\begin{equation}\label{eq:u_appr}
 u_{\text{opt}}(\tau) = \Pi(\tau)[P^{{\star}}(\tau)P(\tau)u_1(\tau)+P^{{\star}}(\tau)\alpha_d]
\end{equation}
for all $\tau \in [0,T]$, with $P \in C([0,T];U^*)$ defined by
\begin{equation}\label{eq:P_R}
 P(\tau)\: w \mapsto \la \tilde{p}(\tau),Bw\ra_{H},\quad \tau \in [0,T],
\end{equation}
and with 
\begin{equation}\label{eq:Pi_R}
\Pi(\tau) = (P^{{\star}}(\tau)P(\tau)+R^{{\star}})^{-1}, \quad \tau \in [0,T].
\end{equation}
Plugging \eqref{eq:discrete-Ansatz} into \eqref{eq:P_R}, we get

\begin{equation}
 P(\tau) \: \sum_{i=1}^M \hat{u}_{i}(\tau)\psi_i \mapsto \la \tilde{p}(t),B \sum_{i=1}^M \hat{u}_{i}(\tau)\psi_i\ra_{H}, \quad \tau \in [0,T]
\end{equation}

\begin{equation}
\la \sum_{i=1}^K \hat{p}_{i}(\tau)\kappa_i,B \sum_{i=1}^M \hat{u}_{i}(\tau)\psi_i\ra_{H} = \sum_{i=1}^K \sum_{j=1}^M \hat{p}_{i}(\tau) \hat{u}_{j}(\tau) \la \kappa_i,B \psi_j\ra_{H} = \hat{p}^\top (\tau) \tilde{B}^\top\hat{u}(\tau),
\end{equation}
thus, we define the following approximation operator
\begin{equation}\label{eq:P_appr}
 P_h(\tau)\: \hat{u}(\tau) \mapsto \hat{p}^\top (\tau) \tilde{B}^\top\hat{u}(\tau), \quad \tau \in [0,T].
\end{equation}
Applying \eqref{eq:discrete-Ansatz} to the variational form of the implementation of operator $R$, we get $\tilde{R}$
\begin{equation*}
\begin{aligned}
& w(\tau) = Ru(\tau) \Leftrightarrow \la w(\tau), v \ra_U = \la Ru(\tau), v \ra_U, \quad \forall v \in U,\quad\tau \in [0,T],\\
& \la w(\tau), \psi_j \ra_U = \sum_{i=1}^M \hat{u}_{i}(\tau) \la R \psi_i, \psi_j \ra_U, \quad \forall j=1,...,M,\quad \tau \in [0,T],\\
& \la w(\tau), \psi_j \ra_U = (\tilde{R}_{j1}, ..., \tilde{R}_{jM})\hat{u}(\tau) , \quad \forall j=1,...,M,\quad \tau \in [0,T].
\end{aligned}
\end{equation*}
Plugging \eqref{eq:P_appr} and $\tilde{R}$ into \eqref{eq:Pi_R} we get
\begin{equation}\label{eq:Pi_appr}
\Pi_h(\tau) = (\tilde{B} \hat{p}(\tau)\hat{p}^\top (\tau) \tilde{B}^\top + \tilde{R}^{\top})^{-1}, \quad \tau \in [0,T].
\end{equation}  
Now plugging \eqref{eq:P_appr} and \eqref{eq:Pi_appr} into \eqref{eq:u_appr} and using that $(\tilde{\Lambda}+\tilde{R}^\top)$
is invertible as $\tilde{\Lambda}$ being symmetric positive semidefinite and $\tilde{R}$ being symmetric positive definite, the result follows.
\end{proof}

The alternative early-lumping approach is to apply the SAC principle from \cite{AnsariMurphey2016} directly to the discretized problem \eqref{eq:J_1_discrete} subject to \eqref{eq:forward_discrete}. With $\tilde{R}$ as in Proposition~\ref{prop:optimizediscretize} a SAC action $\bar{u}_{\text{opt}}(\tau)$ is then chosen as 
\begin{equation*}
 \bar{u}_{\text{opt}}(\tau) := \mbox{argmin}_{\bar{u} \in U}~\hat{l}_2(\bar{u};\tau):=\frac12\left[ \frac{d\hat{J}_1}{d\lambda^+}(\tau,\bar{u})-\alpha_d \right]^2+\frac12\la \bar{u},\tilde{R} \bar{u}\ra_{\RR^M}.
\end{equation*}
For the discretized reference control $\hat{u}_1$ as in Proposition~\ref{prop:optimizediscretize}, we then get 
\begin{equation}\label{eq:soll2problem_discrete2}
 \bar{u}_{\text{opt}}(\tau)=\left(\Lambda(t)+\tilde{R}^\top\right)^{-1}\left[ \Lambda(t)\hat{u}_1(\tau)+\alpha_d \mathcal{B}^\top ((\mathcal{M}^{-1} )^{\top} \hat{\rho}(\tau))\right],
\end{equation}
where $\Lambda(\tau)=\mathcal{B}^\top((\mathcal{M}^{-1})^\top\hat{\rho}(\tau))\hat{\rho}(\tau)^\top \mathcal{M}^{-1} \mathcal{B}$, $\hat{\rho}$ solves the backward ODE
\begin{equation}\label{eq:adjoint_discrete2}
\dot{\hat{\rho}}(t)=-\mathcal{A}^{\top}(\mathcal{M}^{-1})^{\top}\hat{\rho}(t)-(\hat{l}_1)_{\hat{y}}(\hat{y}(t)),\quad \hat{\rho}(T)=\hat{m}_{\hat{y}}(\hat{y}(T))
\end{equation}
and $\hat{y}$ is the solution of \eqref{eq:forward_discrete}. Only under certain assumptions it holds that the two approaches yield the same control action.

\begin{proposition}\label{prop:galerkin}
Choosing $V_h^1=V^2_h$ with the same basis, i.e., $K=N$ and $\phi_i=\kappa_i$ for $i=1,\ldots,N$, we have
\begin{equation*}
\hat{p}(t)=(\mathcal{M}^{-1})^{\top}\hat{\rho}(t),\quad t \in [0,T].
\end{equation*}
In particular, $\hat{u}_{\text{opt}}$ obtained from \eqref{eq:soll2problem_discrete} and $\bar{u}_{\text{opt}}$ obtained from \eqref{eq:soll2problem_discrete2} coincide.
\end{proposition}
\begin{proof}
Choosing $V_h^1=V^2_h$ with the same basis functions, we have $\mathcal{A}^\top=\tilde{A}$, $\mathcal{B}^\top=\tilde{B}$ and $\mathcal{M}^\top=\tilde{M}$. With that, the result follows from \eqref{eq:adjoint_discrete} and \eqref{eq:adjoint_discrete2} as well as \eqref{eq:soll2problem_discrete}
and \eqref{eq:soll2problem_discrete2}.
\end{proof}

\begin{remark}\label{rem:remark2}
The motivation to use different Ansatz functions for the state and the adjoint state comes from the reason that the adjoint variable $p$ can admit more regularity than the state $y$. Thus, it is meaningful to keep different Ansatz functions. A detailed discussion can be found in (\cite{Hinze_Pinnau}, Section 3.2).
\end{remark}

We note that error analysis can be applied to \eqref{eq:soll2problem_discrete} in order to estimate the error of the finite-dimensional SAC control action in terms of data for the stage problem~\eqref{eq:stageproblem}. For appropriate techniques concerning parabolic problems, see, e.g., \cite{MeidnerVexler2008}.
\subsection{Numerical results}\label{subsec:numerical}
For our numerical study of SAC for the problem \eqref{eq:1dproblem} we choose the parameters provided in Table~\ref{tab:parameter}. Our numerical implementation uses the Galerkin approximation presented in Subsection~\ref{subsec:galerkin}. Furthermore, we use a Gelfand triple with $H = L^2(\Omega), V = H^1_0(\Omega)$, and $V^* = H^{-1}(\Omega)$. We work under the hypothesis of Proposition~\ref{prop:galerkin} and choose piecewise linear functions for the finite-dimensional state and adjoint subspaces $V_h = V^{*}_h$. For the finite-dimensional control subspace $U_h$ we take piecewise constant functions on an equidistant grid with mesh size $h=0.01$. The resulting ODEs are solved numerically using the implicit Euler method in time at the sampling times $t_s$. Unlike in the original SAC algorithm  in \cite{AnsariMurphey2016} and the generalization in Section~\ref{sec:framework}, but in order to make the numerical results comparable to the theoretical results in the Subsection~\ref{subsec:heateq}, we consider the calculation time $t_{calc} = 0$ and a fixed control application time $\bar{\lambda} = t_s$. However, we note that our numerical experiments reveal that small $t_{calc} > 0$ time stepping for $\bar{\lambda}$ using line search does not change the results qualitatively.
\begin{table}[h]
\begin{tabular}{lll}
\textbf{Parameter} & \textbf{Meaning} & \textbf{Value}\\
\hline \hline
 $L$ & length of the spatial 1-D area & $1$\\
 $\mu$ & instability constant & $1.35 \pi^2$\\
 $y_0(x)$ & initial value for the temperature profile & $\frac{1}{5}\sin\left({\pi}x\right)$\\
 $\bar{q}$ & weight constant for matrix $Q$ & $10$\\
 $\beta$ & constant & $1.6$\\
 $t_s$ & sampling time step & $0.1$\\
\end{tabular}
\caption{Problem parameters chosen for the numerical study concerning the reaction-diffusion problem \eqref{eq:1dproblem} and the chosen stage costs \eqref{eq:1dcost}}
\label{tab:parameter}
\end{table}
\subsubsection{Full domain control.}
We recall that the solution of \eqref{eq:1dproblem} without control, i.e., $u(t)=u_1(t)=0$, is exponentially unstable. Our numerical stabilization results for the full domain control are reported in Figure~\ref{fig:different_gamma_and_T}. The Subfigures~(A) and ~(B) illustrate the performance of the finite-dimensional SAC controller for driving the state towards the unstable equilibrium $y_d=0$ with the choice $T=1$ and $\gamma= -0.5$. In Subfigure~(C), we see that smaller $\gamma$ lead to faster stabilization. However, due to our fixed implicit time stepping, the rate is limited by overshooting which becomes visible in our example for $\gamma=-1$ in $t=0.5$. Subfigure~(D) shows that a longer time horizon $T$ leads to a smaller error of the state in the $L_2$-norm. We can observe that the stabilization rate is actually exponential until the error drops below a small constant that depends on the chosen sampling time.
\begin{figure}[!t]
\captionsetup{width=0.8\linewidth}
  \begin{subfigure}[b]{0.49\linewidth}
    \includegraphics[width=\linewidth]{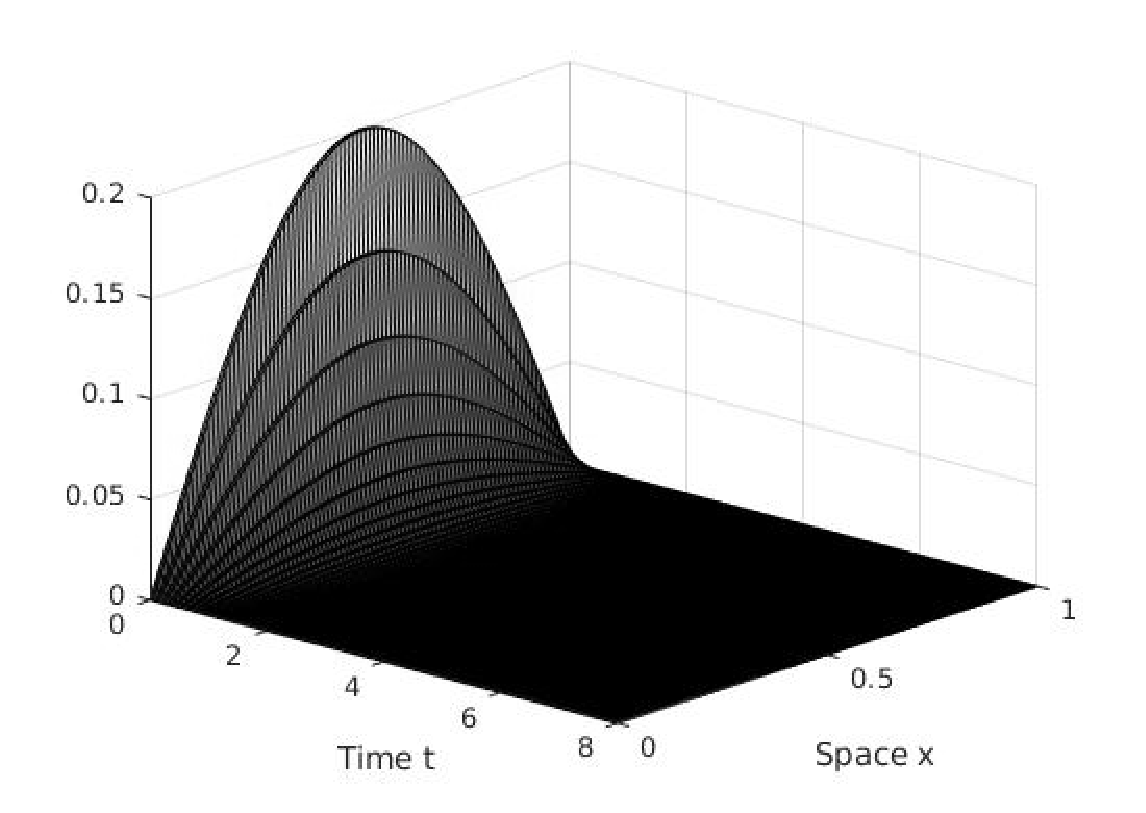}
    \caption{}
  \end{subfigure}
  \begin{subfigure}[b]{0.49\linewidth}
    \includegraphics[width=\linewidth]{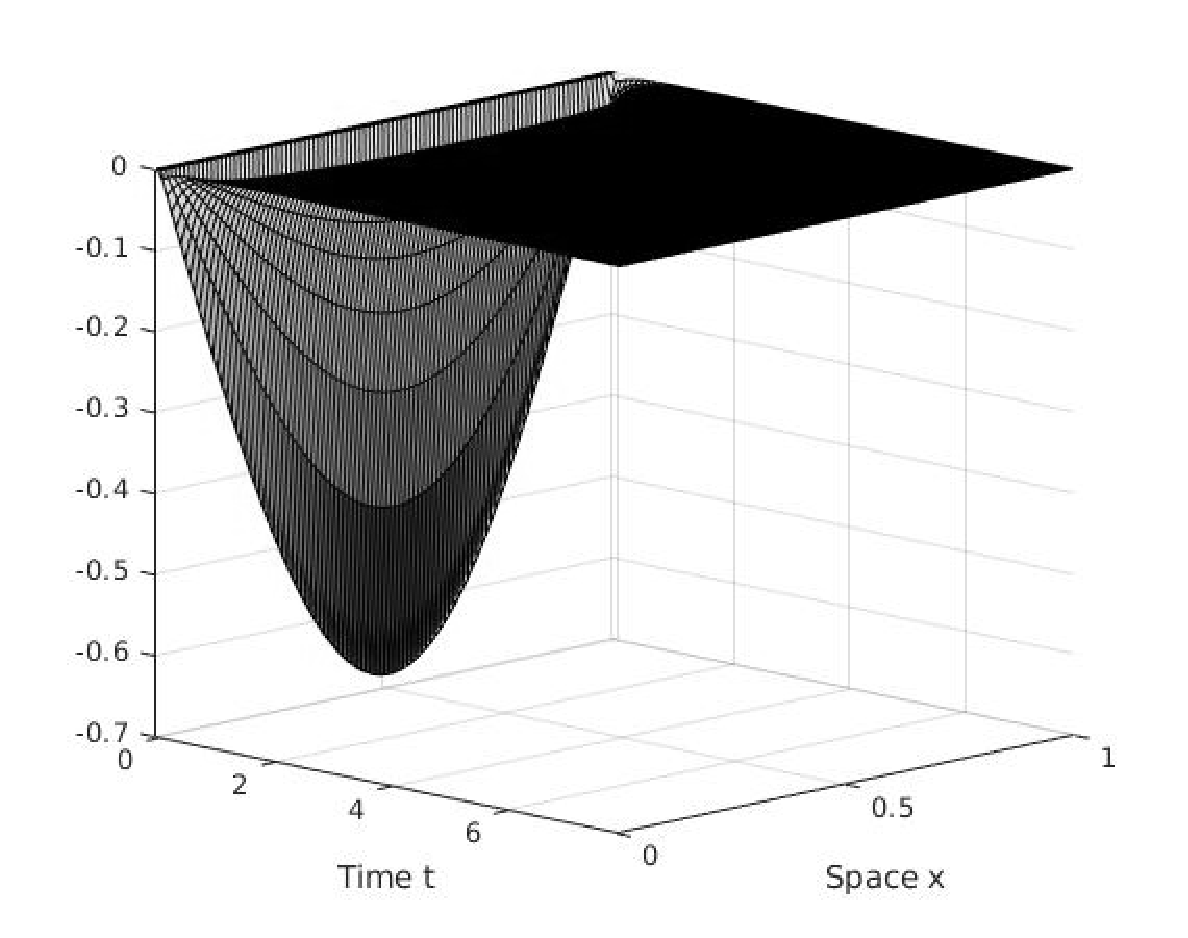}
    \caption{} 
  \end{subfigure}\\
  \begin{subfigure}[b]{0.49\linewidth}
    \includegraphics[width=\linewidth]{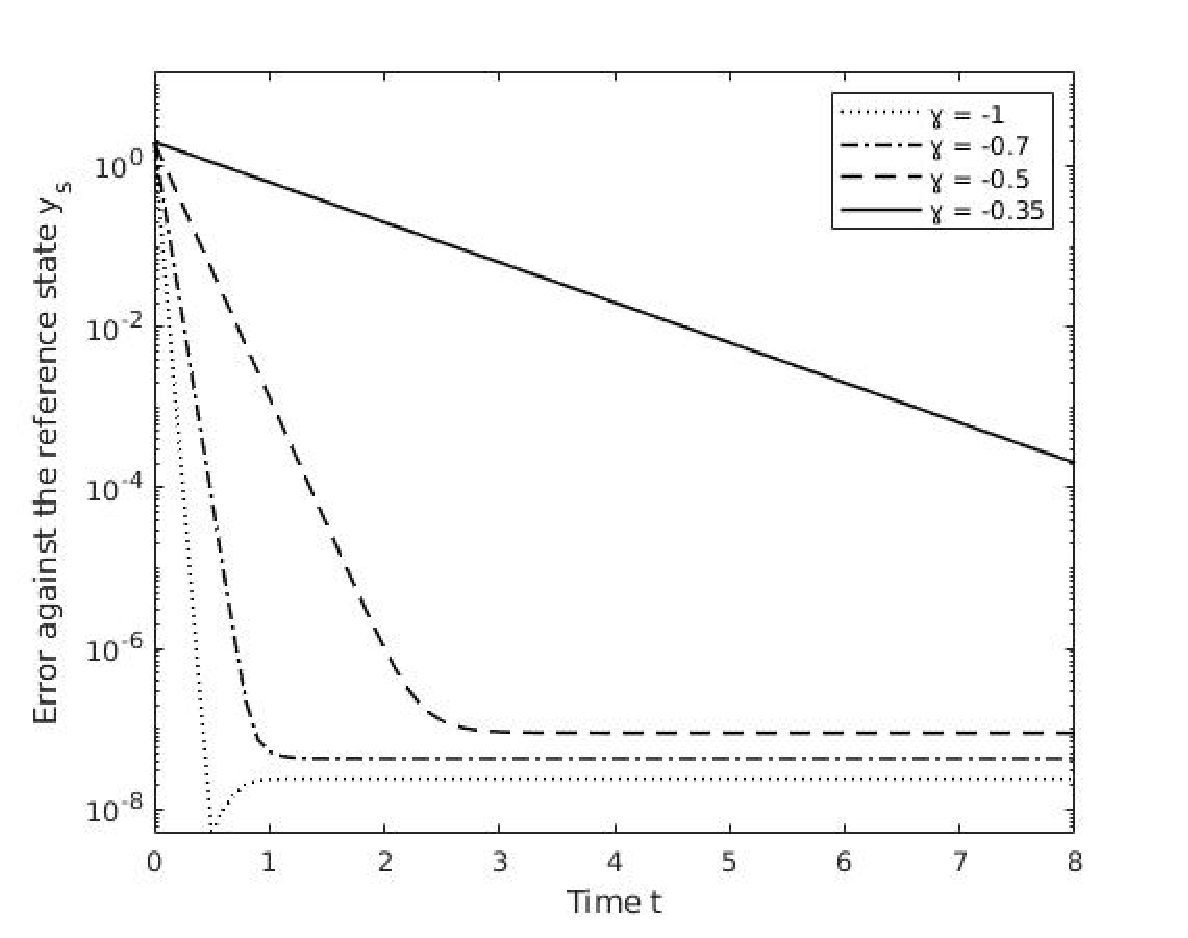} 
    \caption{}
  \end{subfigure}
  \begin{subfigure}[b]{0.49\linewidth}
    \includegraphics[width=\linewidth]{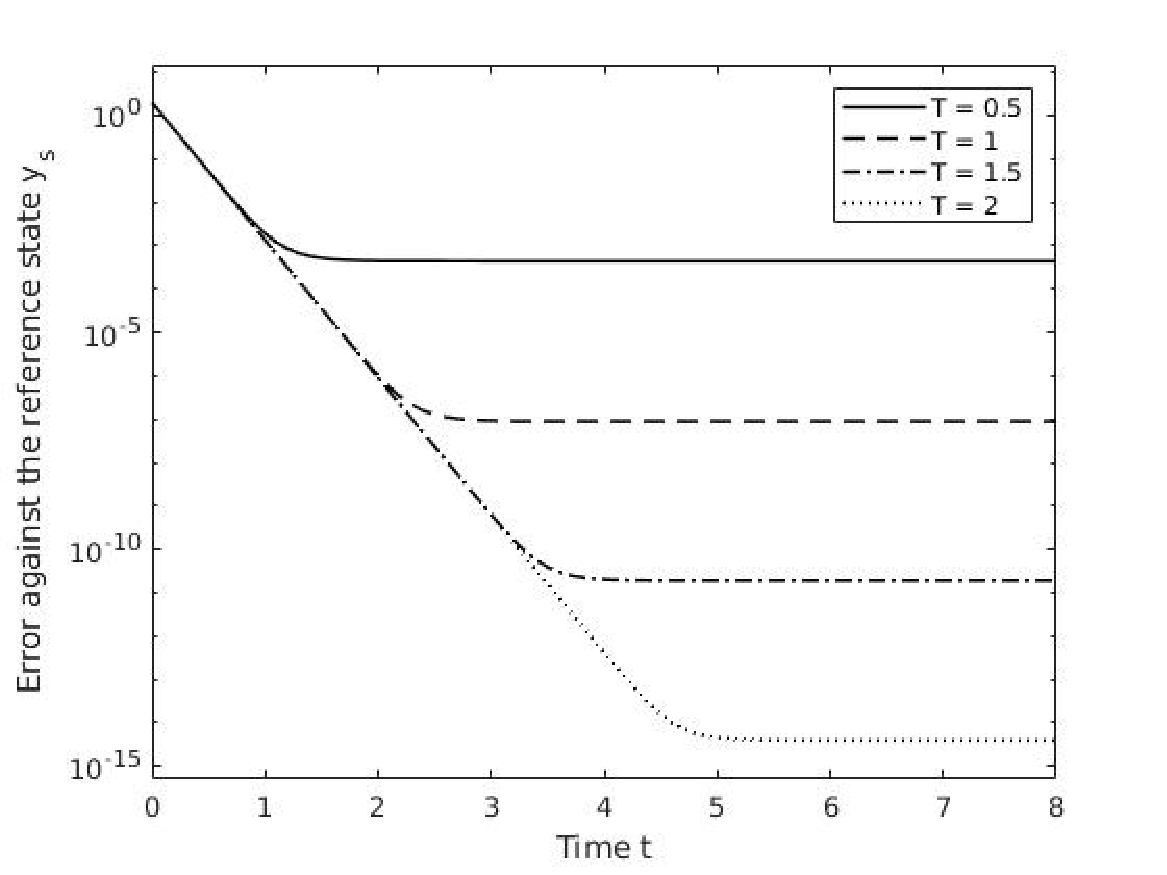}
    \caption{}
  \end{subfigure}
  \caption[0.2\linewidth]{Numerical results for the unstable heat equation~\eqref{eq:1dproblem} with full domain control. (A) shows the stabilized state using SAC feedback (for $T = 1$, $\gamma = -0.5$) and (B) shows the corresponding control. (C) shows the error of the state in the $L_2$-norm for different parameters $\gamma$ (for $T=1$) and (D) shows this error for different time horizons $T$ (for $\gamma = -0.5$)}
  \label{fig:different_gamma_and_T}\vspace*{-9pt}
\end{figure}

\subsubsection{Subdomain control.}
Coming back to more general settings, the results for control on subdomain $\Omega_c = [0,\infty) \times (0.5, 0.9)$ are reported in Figure~\ref{fig:subdom_control}. Furthermore, in order to check robustness of the closed-loop system after incorporating controls obtained by SAC with respect to some random disturbance, we took in stage problem $\mu = 1.2 \pi^2$ with maximum of $10 \%$ random disturbance on each step during simulation. All other parameters are the same as before. We can observe on the Subfigures~(A), ~(B) and ~(C) that even subdomain control found by SAC gives us the exponential  stabilization rate. These results are promising for future work towards boundary control. On Subfigure~(C) it is possible to see that control obtained by SAC drive the system to some sort of equilibrium interval in which the error remains after reaching a certain small enough value. This shows, that the system is robust with respect to small disturbances which is promising in order to use SAC in the direction of robust optimization. 
\begin{figure}[!t]
\captionsetup{width=0.8\linewidth}
	\begin{subfigure}[b]{0.49\linewidth}
	  \includegraphics[width=1\linewidth]{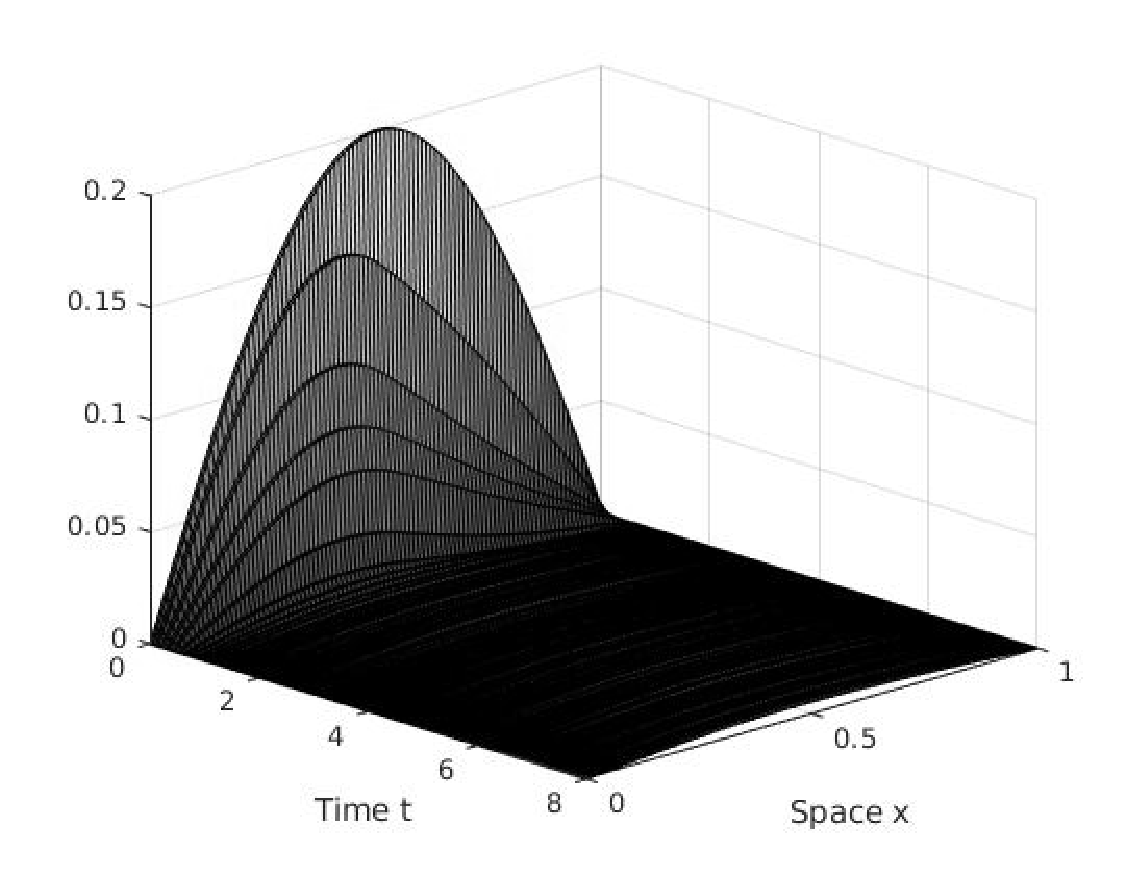}
	  \caption{}
	\end{subfigure}
	\begin{subfigure}[b]{0.49\linewidth}
	  \includegraphics[width=1\linewidth]{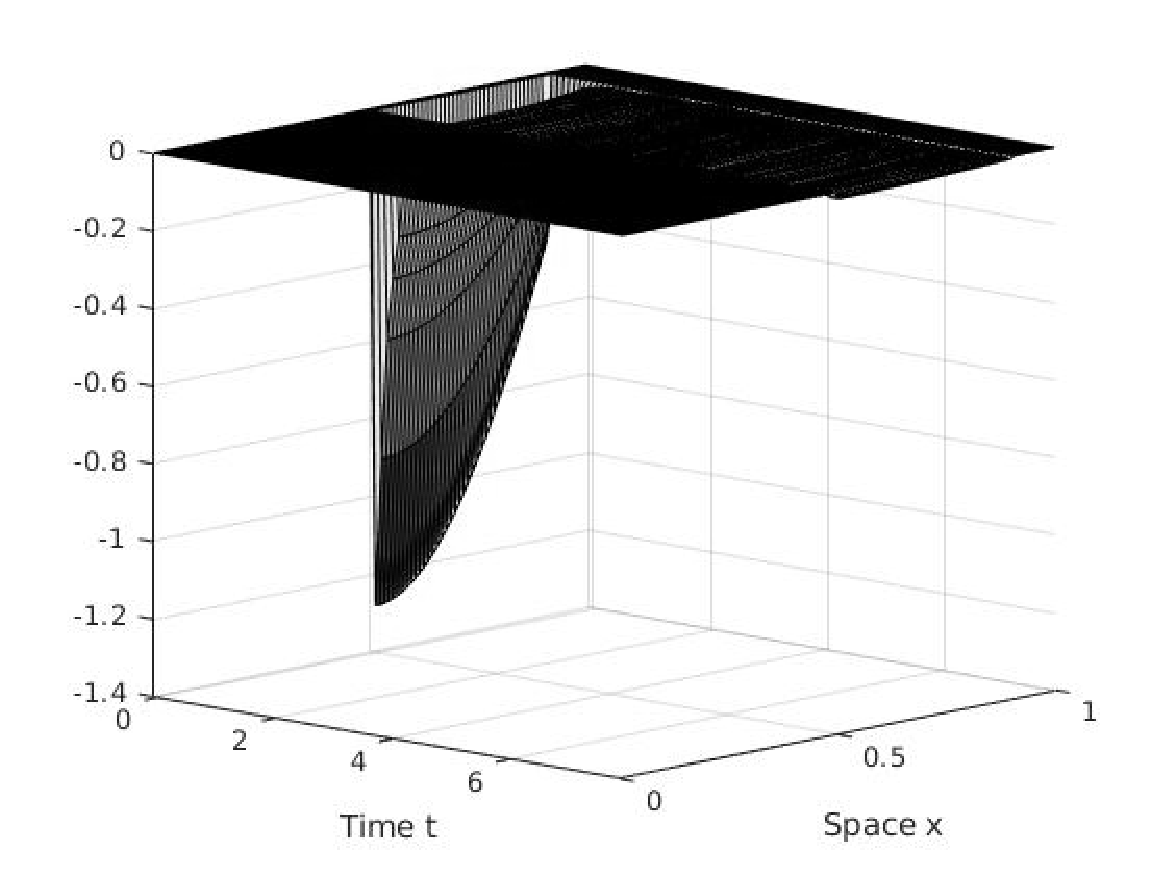}
	  \caption{}
	\end{subfigure}
	\begin{subfigure}[b]{0.49\linewidth}
	  \includegraphics[width=1\linewidth]{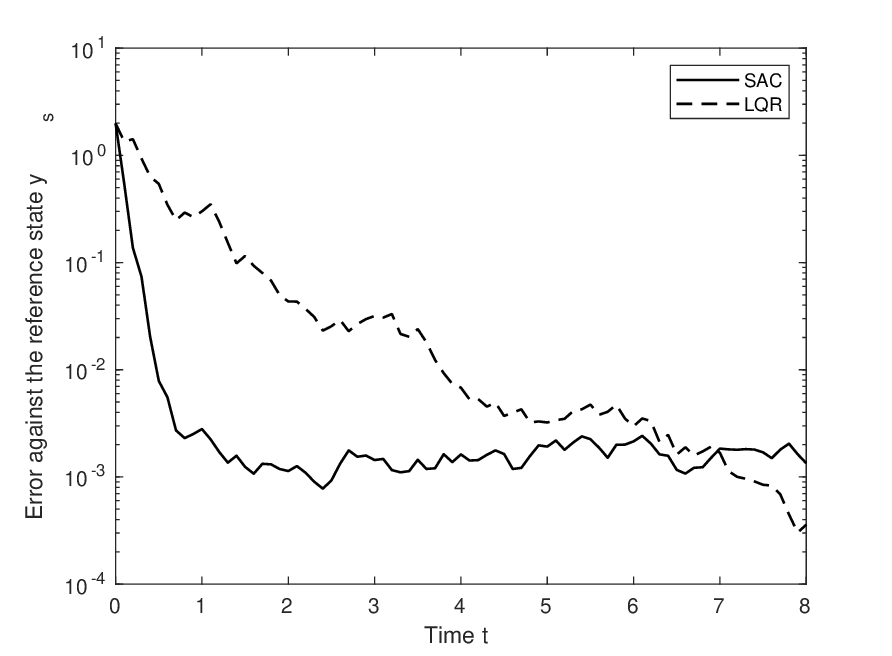}
	  \caption{}
	\end{subfigure}
	  \caption{Numerical results for the unstable heat equation~\eqref{eq:1dproblem} with control on subdomain $\Omega_c = [0,\infty) \times (0.5, 0.9)$ and instability constant $\mu = 1.2 \pi^2$ with $10 \%$ disturbance. (A) shows the stabilized state using SAC feedback (for $T = 1$, $\gamma = -1$) and (B) shows the corresponding control.
   (C) shows the error of the state in the $L_2$-norm for $T=1$, $\gamma= -1$, for SAC and LQR}
  \label{fig:subdom_control}
\end{figure}
\subsubsection{Partial observation.}
Along with control on subdomain we also look on the results of control with partial observation. This means that we will change weight constant $\bar{q}$ for matrix $Q$ in the cost function into $q^{*} = \chi_{(a,b)}(x)\bar{q}$, for some given interval $(a,b)$ where we carry observations. The results for different intervals of observations are given in Figure~\ref{fig:part_obs_control}. On the Subfigures~(A), ~(B) and ~(C) we can observe, that control found by SAC with partial observation remain the exponential stabilization rate shown before. On Subfigure~(C) we can see, that the bigger region of observation will result in better stabilization.
\begin{figure}[!t]
\captionsetup{width=0.8\linewidth}
	\begin{subfigure}[b]{0.49\linewidth}
	  \includegraphics[width=1\linewidth]{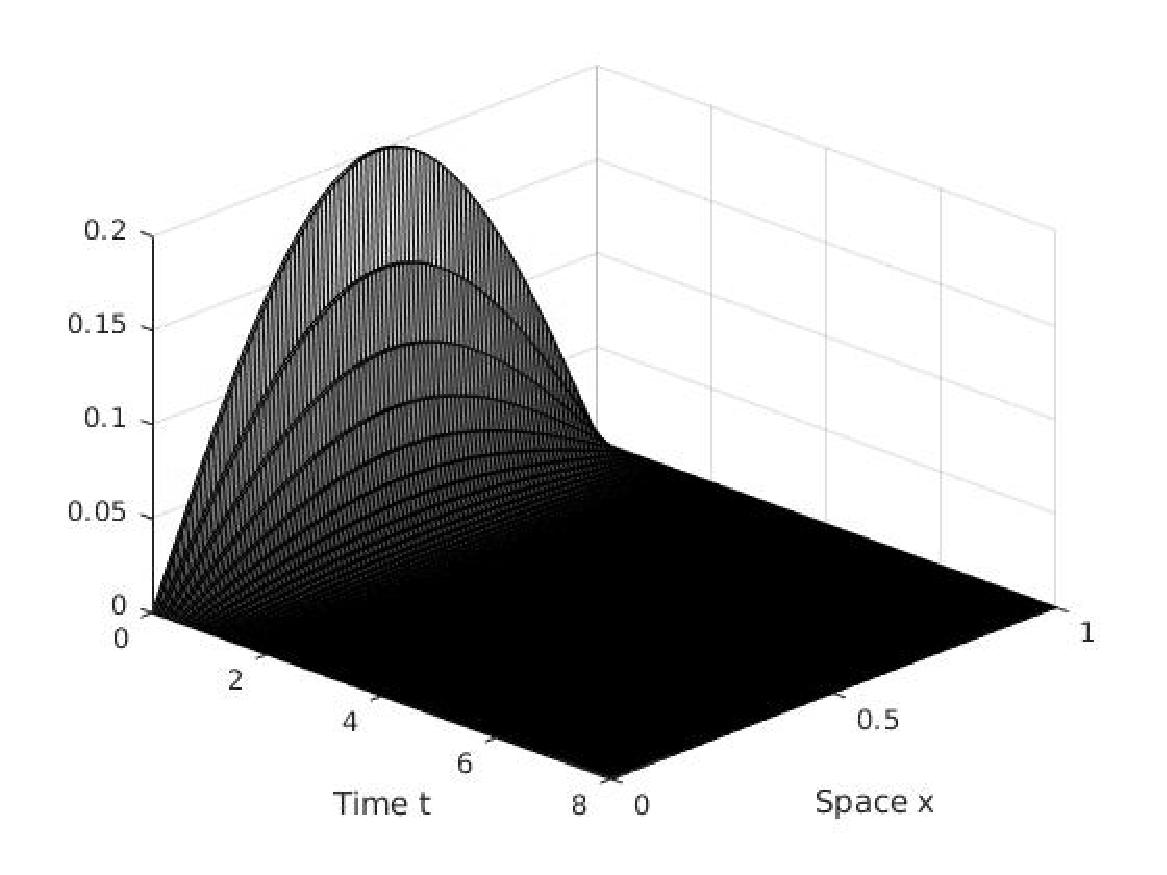}
	  \caption{}
	\end{subfigure}
	\begin{subfigure}[b]{0.49\linewidth}
	  \includegraphics[width=1\linewidth]{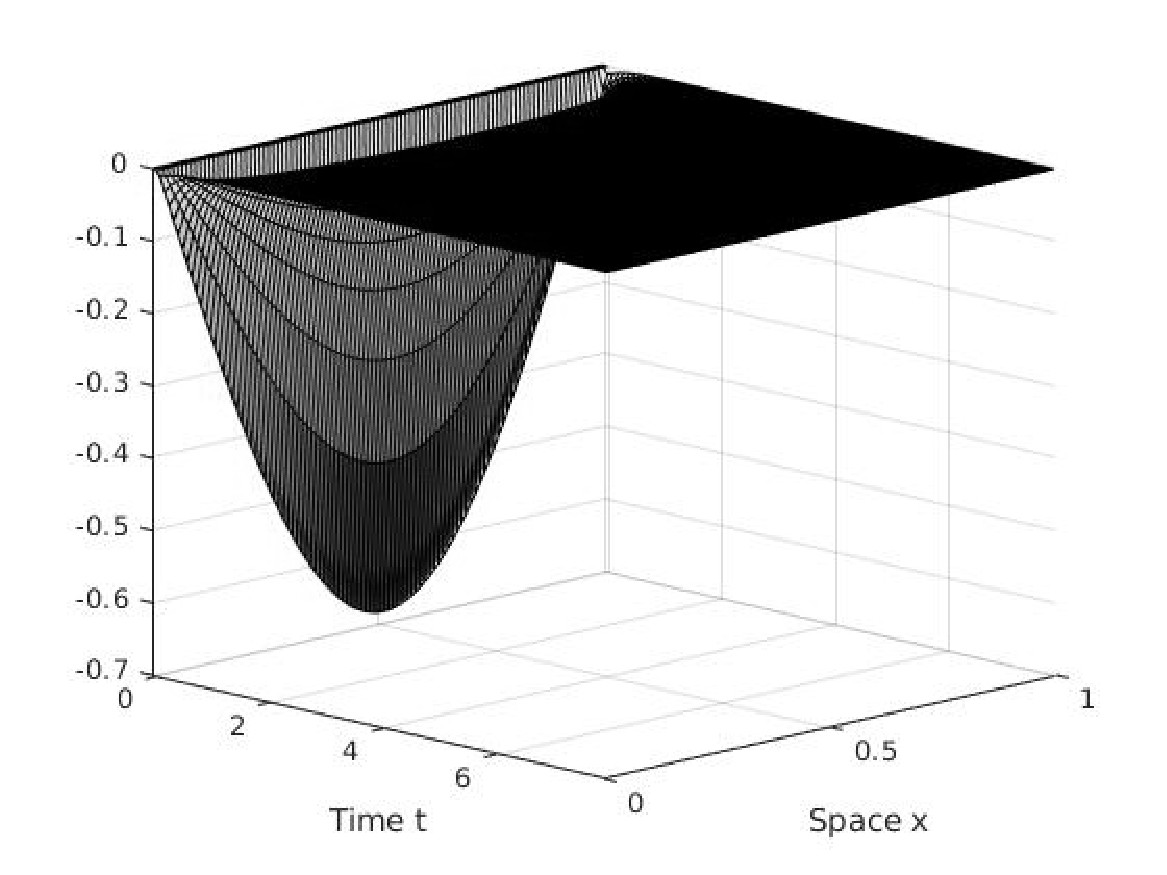}
	  \caption{}
	\end{subfigure}
	\begin{subfigure}[b]{0.49\linewidth}
	  \includegraphics[width=1\linewidth]{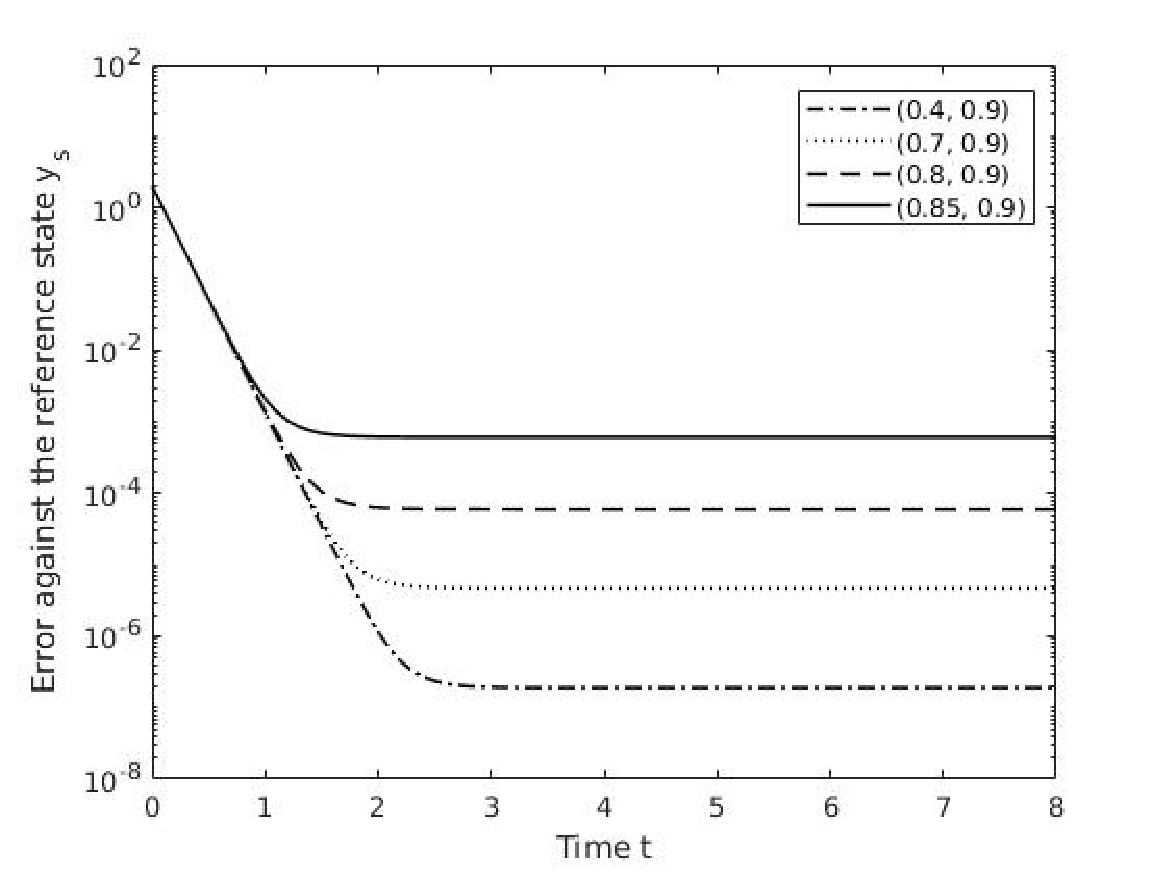}
	  \caption{}
	\end{subfigure}
	  \caption{Numerical results for the unstable heat equation~\eqref{eq:1dproblem} with control with partial observation on $\Omega_c = [0,\infty) \times (0.7, 0.9)$. (A) shows the stabilized state using SAC feedback (for $T = 1$, $\gamma = -0.85$) and (B) shows the corresponding control.
   (C) shows the error of the state in the $L_2$-norm for $T=1$, $\gamma= -0.85$, for different observation intervals}
  \label{fig:part_obs_control}
\end{figure}
\subsubsection{Comparison with LQR method.}
Finally we compare SAC with the standard linear-quadratic regulator (LQR) design for finding optimal control solution u. For the case of subdomain control with some disturbance (as described before), the result is given in Figure~\ref{fig:subdom_control}, Subfigure~(C). It is possible to see, that in the case of SAC the error is reaching acceptable error faster, then the one using LQR. Furthermore, another big improvement is time of work. While standard LQR took 20.03s to find a solution, SAC finds a solution in just 0.33s, which is a great improvement and a prominent direction towards real-time decision-making. Similar results are obtained when comparing LQR and SAC for full-domain control. Even though LQR drive the solution closer to the desired one (error could be much smaller), SAC drives the solution towards acceptable error faster and took significantly less time to perform.
\section{Concluding remarks}\label{sec:conclusion}
Our investigations show that sequential action control (SAC) is a promising framework for control and stabilization of PDE-dynamical problems. 
It is well suited to include measurements in an online fashion for example in order to account for uncertainties. As a particular variant of a moving horizon method 
and in contrast to classical model predictive control, the control synthesis can be done without solving a dynamic optimization problem. This
makes it very easy to implement the controller. Moreover, from this we can see that the control principle can be easily extended to piecewise 
linear switched systems and hence arbitrarily close approximations of nonlinear evolutions. 

Here we have taken a first step towards a qualitative analysis of this control principle in a Hilbert space framework with distributed control. As in the case of problems with 
ordinary differential equations, the stability analysis for the important case of quadratic stage costs turns out to be closely related to 
LQR-theory. We have been applying this prototypically for stabilization of a reaction-diffusion process. 

Possible directions for future work is the stability analysis, including decay rates, for problems including boundary control, potentially using the results obtained in~\cite{Bound_control}, partial state observation, other versions of \eqref{eq:l2problem}, hyperbolic dynamical systems as well as complying with state constraints.

\section*{Funding}
This work was funded by the Deutsche  Forschungsgemeinschaft  (DFG,  German  Research  Foundation) [Project-ID - 239904186] – TRR 154 - subproject A03.
\markboth{Y. Brodskyi, F.~M. Hante, A. Seidel}{\rm STABILIZATION OF PDES BY SEQUENTIAL ACTION CONTROL}
\vspace*{6pt}
\bibliographystyle{abbrv}

\end{document}